\documentclass[reqno]{amsart}
\usepackage[utf8]{inputenc}
\usepackage{amsmath,amsthm,amssymb}
\usepackage{hyperref}
\date{\today}
\usepackage{paralist}
\usepackage{cite}

\theoremstyle{definition}
\newtheorem{theorem}{Theorem}[section]

\newtheorem{theorem*}{Theorem}
\newtheorem{remark}[theorem]{Remark}
\newtheorem{lemma}[theorem]{Lemma}
\newtheorem{corollary}[theorem]{Corollary}
\newtheorem{example}[theorem]{Example}
\newtheorem{proposition}[theorem]{Proposition}

\newcommand{\tr}{\operatorname{tr}}
\newcommand{\nul}{\operatorname{nul}}

\theoremstyle{definition}
\theoremstyle{definition}
\newtheorem{definition}[theorem]{Definition}
\newtheorem*{definition*}{Definition}

\usepackage{hyperref}
\usepackage{comment}
\begin{document}

\newcommand{\Pos}{Pos}
\renewcommand{\Im}{\operatorname{Im}}

\newcommand{\Hess}{\operatorname{Hess}}
\newcommand{\Id}{\operatorname{Id}}
\newcommand{\Sym}{\operatorname{Sym}}
\newcommand{\vol}{\operatorname{vol}}
\newcommand{\Hom}{\operatorname{Hom}}

\newcommand{\usc}{\operatorname{usc}}

\newcommand{\argmin}{\operatorname{argmin}}
\newcommand{\diag}{\operatorname{diag}}

\renewcommand{\Pos}{\operatorname{Pos}}
\newcommand{\calPos}{\mathcal P}
\newcommand{\Int}{\operatorname{Int}}
\setcounter{tocdepth}{1}

\begin{abstract}
We extend Prekopa's Theorem and the Brunn-Minkowski Theorem from  convexity to $F$-subharmonicity.  We apply this to the interpolation problem of convex functions and convex sets introducing a new notion of ``harmonic interpolation" that we view as a generalization of Minkowski-addition.
\end{abstract}

\title[{Interpolation, Prekopa and Brunn-Minkowski for $F$-subharmonicity} ]{Interpolation, Prekopa and Brunn-Minkowski\\ for $F$-subharmonicity} 
\author{Julius  Ross and David Witt Nystr\"om}

\maketitle
\tableofcontents
\numberwithin{equation}{section}

\section{Introduction}

\subsection{Interpolation of Convex Sets}

Let $A_0$ and $A_1$ be two bounded convex subsets of $\mathbb R^m$.  There is a natural interpolation between these  obtained by setting
\begin{equation}A_t : = tA_1 + (1-t)A_0 \text{ for } t\in [0,1]\label{eq:basicinterpolation}\end{equation}
where  $A+B: = \{ a+b : a\in A,b\in B\}$ is  Minkowski-addition, and the scaling is defined by $tA: = \{ta: a\in A\}$.  The celebrated Brunn-Minkowski Theorem implies that with this interpolation the map
$$t\mapsto -\log \vol(A_t)$$
is convex.  

It is natural to ask if anything similar can be said for the interpolation of infinite families of convex subsets of $\mathbb R^m$.  So suppose $\Omega\subset \mathbb R^n$ is open and bounded with smooth boundary, and for each $\tau\in \partial \Omega$ we are given a compact convex subset $A_\tau\subset \mathbb R^m$ that varies continuously with $\tau$.     The question is how can we interpolate this family to give a compact convex set  $A_x\subset \mathbb R^m$ for each $x\in \Omega$ in a nice way?  

A naive answer would be to simply take the fibers of  \begin{equation}\text{convex hull}\left(\bigcup_{\tau\in \partial \Omega} \{\tau\}\times A_{\tau}\right)\subset \overline{\Omega}\times \mathbb R^m\label{eq:convexhull}\end{equation}  but, unless one assumes that $\Omega$ is strictly convex, it will in general not be the case that this agrees with the given sets $A_\tau$ over $\partial \Omega$.  Furthermore even if $\Omega$ is strictly convex the interpolation \eqref{eq:convexhull} does not respect Minkowski addition of the boundary data.

Our proposed solution is different and can be described using the harmonic measure $d\mu_x$ on $\partial \Omega$  taken with respect to $x\in \Omega$.

\begin{definition}
Define
$$A_x: = \int_{\partial \Omega} A_\tau d\mu_x(\tau) \text{ for } x\in \Omega,$$
and  call the family $\{A_x\}_{x\in \Omega}$ the \emph{harmonic interpolation}.
\end{definition}

This harmonic interpolation involves a set integral which may be unfamiliar to the reader but poses no difficulties in the setup we will be considering.  For instance, one may think of this as a Riemann integral which is a limit of Riemann sums taken under the operation of Minkowski-addition.    We will prove the following: 

\begin{theorem}[Brunn-Minkowski Inequality for Harmonic Interpolation]\label{thm:introBM}
With the harmonic interpolation,  the map
$$x \mapsto -\log \vol(A_x)$$
is subharmonic on $\Omega$.
\end{theorem}

Observe that in the special case $\Omega = [0,1]$ the harmonic measure with respect to $t\in \Omega$ is $$d\mu_t = (1-t)\delta_0 + t\delta_1$$ where $\delta_p$ is the Dirac measure at $p$, so $A_t = \int_{\partial \Omega} A_x d\mu_t(x) = (1-t)A_0 + tA_1$ as in \eqref{eq:basicinterpolation}.   Thus our theorem generalizes the classical Brunn-Minkowski Theorem since being subharmonic in one real variable is the same as being convex.

This harmonic interpolation has other amenable features namely it depends linearly on the boundary data (with respect to Minkowski-addition) and is compatible with affine transformations of $\mathbb R^m$.   These two features do not characterize the harmonic interpolation uniquely for the same would hold if one replaces the harmonic measure on $\partial \Omega$ with any other measure, but one can characterize it through a mean value property analogous to that of harmonic functions.  We expect that the harmonic interpolation also has nice ``regularity properties" analogous to that of elliptic regularity, but we will not consider that further in this paper.


\subsection{Interpolation of Convex Functions}

The functional analog of the Brunn-Minkowski Theorem is given by  Prekopa's Theorem, a version of which can be stated as follows: if $\psi(t,y)$ is a convex function on $[0,1]\times \mathbb R^m$ then the function 
$$t\mapsto -\log \int_{\mathbb R^m} e^{-\psi(t,y)} dy$$
is convex. And just as the Brunn-Minkowski Theorem has implications for the interpolation of convex sets,  Prekopa's Theorem has implications for the interpolation of convex functions.

To discuss this, suppose that $\phi_0$ and $\phi_1$ are two convex functions on $\mathbb R^m$.  The naive interpolation between these is the combination $t\phi_1 + (1-t)\phi_0$ for $t\in [0,1]$, but this has the disadvantage that $(t,y)\mapsto t\phi_1(y) + (1-t)\phi_0(y)$ typically fails to be convex on $[0,1]\times \mathbb R^m$.  So instead consider the interpolation
$$\phi_t : = ((1-t)\phi_0^* + t \phi_1^*)^*$$
where the star denotes the Legendre transform
$$\phi^*(u) = \sup_{y\in \mathbb R^m} \{ y\cdot u - \phi(y) \}.$$
As the Legendre transform is an involution on convex functions, $\phi_t^* = (1-t)\phi_0^* + \phi_1^*$ and it is easy to check that $$\psi(t,y): = \phi_t(y)$$ is convex so Prekopa's Theorem applies.

Once again we can consider interpolating infinitely many convex functions $\phi_\tau$ on $\mathbb R^m$, parametrized by points $\tau$ in the boundary of a bounded domain $\Omega\subset \mathbb R^n$ with smooth boundary.  We assume always that $(\tau,y)\mapsto \phi_{\tau}(y)$ is continuous.   Define
$$ \phi_x = \left(\int_{\partial \Omega} \phi_\tau^* d\mu_x(\tau) \right)^* \text{ for } x\in \Omega$$
where again $d\mu_x$ is the harmonic measure on $\partial \Omega$ with respect to the point $x$.   

\begin{definition}
Set
$$\psi(x,y) = \phi_x(y) \text{ for } x\in \Omega, y\in \mathbb R^m$$
which we refer to as the \emph{harmonic interpolation} of the data $\{\phi_{\tau}\}_{\tau \in \partial \Omega}$.   
\end{definition}

Now $\psi$ will in general not be a convex function so we cannot apply Prekopa's Theorem as stated.  However $\psi$ has another form of positivity, namely:
\begin{enumerate}
\item For each $x\in \Omega$ the function $y\mapsto \psi(x,y)$ is convex and
\item For each affine function $\Gamma:\mathbb R^n\to \mathbb R^m$ the function $x\mapsto \psi(x,\Gamma(x))$ is subharmonic on $\Omega$.
\end{enumerate}


\begin{theorem}[Prekopa's Theorem: Version I]\label{thm:introPrekopa}
Suppose  $\psi:\Omega\times \mathbb R^m\to \mathbb R\cup \{-\infty\}$ is upper-semicontinuous and satisfies (1) and (2).  Then the function $$x\mapsto -\log \int_{\mathbb R^n} e^{-\psi(x,y)} dy$$ is subharmonic on $\Omega$. 
In particular this applies to the harmonic interpolation of a family of convex functions.
\end{theorem}

We are aware of two different proofs of this statement.  First there is direct proof that uses the classical Prekopa Theorem (which will appear in a forthcoming companion paper \cite{ross_nystron_unpublished}).  Second there is a more flexible proof (given in this paper) that considers this as a special case of a statement that replaces subharmonicity with a more general positive notion described below.\\

There is a close connection between this interpolation of convex functions and the previously discussed interpolation of convex sets.  Given a compact convex $A\subset \mathbb R^m$ the \emph{indicator function} of $A$ is

$$\chi_A : = \left\{ \begin{array}{cc} 0 & \text{on }A \\ \infty & \text{ on } A^c\end{array}\right.$$ and the \emph{support function} of $A$ is
$$h_A(u) : = \sup_{x\in A} \{x\cdot u\},$$
so by definition $\chi_A$ and $h_A$ are the Legendre transform of one another.  Using linearity of the map $A\mapsto h_A$ we will see that the harmonic interpolation of the functions $\{\chi_{A_\tau}\}_{\tau\in \partial \Omega}$ is precisely the indicator function of the harmonic interpolation of the sets $A_x$.   Moreover since 
$$\int_{\mathbb R^m} e^{-\chi_A} dy = \vol(A)$$
our generalization of Brunn-Minkowski's Theorem (Theorem \ref{thm:introBM}) follows from our generalization of Prekopa's Theorem (Theorem \ref{thm:introPrekopa}).

\subsection{Berndtsson's Complex Prekopa Inequality}\label{sec:introBo}

Our version of Prekopa's Theorem also quickly yields the following statement for plurisubharmonic functions, originally due to Berndtsson \cite{Berndtsson_Prekopa}.

\begin{theorem}
Let $X\subseteq \mathbb{C}^n$ be open and assume that $\psi$ is plurisubharmonic on $X\times \mathbb{C}^m$. Let $w_j=x_j+iy_j$ be the coordinates on $\mathbb{C}^m$ and assume that  $\psi(z,x+iy)=\psi(z,x)$.  Then $$z\mapsto -\log\int_{\mathbb{R}^m}e^{-\psi(z,x)}dx$$ is plurisubharmonic in $X$.  
\end{theorem}

There is also a closely related version of this where instead of being independent of the imaginary part of $w$ we instead take $\psi$ to be independent of the argument of $w$.     The link with what we have said so far comes from the fact that the hypotheses imply that $\psi$ satisfies (1) and (2) above so Theorem \ref{thm:introPrekopa} applies. In fact Berndtsson's Theorem is really equivalent to Theorem \ref{thm:introPrekopa} when $\Omega\subset \mathbb C\simeq \mathbb R^2$.

\subsection{The Prekopa and Brunn-Minkowski Theorems for $F$-subharmonicity}

We have stated our results so far in terms of convexity and subharmonicity, but they can be cast in the much wider context of ``generalized subharmonicity".   To discuss this we recall some notions introduced and studied by Harvey-Lawson.

Let $F$ be a closed non-empty proper subset of the set $\Sym^2(\mathbb R^n)$ of symmetric $n\times n$ matrices and $\mathcal P\subset \Sym^2(\mathbb R^n)$ denote the positive semidefinite matrices.  We say that $F$ is a \emph{Dirichlet set} if
$$ A\in F \text{ and } P \in \mathcal P \Rightarrow A+P\in F.$$
A smooth function on an open  $X\subset \mathbb R^n$ is said to be \emph{$F$-subharmonic} if at each point its Hessian lies inside $F$.  

One can extend $F$-subharmonicity to upper-semicontinuous functions $f:X\to \mathbb R\cup \{-\infty\}$ using the so-called viscosity technique.   Call a smooth function $g$ in $\Omega$ a  \emph{test function for $f$ at $x\in \Omega$} if $g$ is smooth and $g\ge f$ near $x$ and $g(x)=f(x)$.  Then we say $f$ is \emph{$F$-subharmonic} if for any test function for $g$ at $x$ the Hessian of $g$ at $x$ lies in $F$.

We refer the reader to \S\ref{sec:backgroundFsubharmonicity} for more background to $F$-subharmonicity including various examples.  For now,  the reader may note that if $F= \mathcal P$ then $F$-subharmonicity becomes the usual notion of convexity.  On the other hand if $F$ is taken to be the set $F_{sub}$ symmetric matrices with positive trace then $F$-subharmonicity reduces to the usual notion of subharmonicity.\\

Now given such a Dirichlet set $F\subset \Sym^2(\mathbb R^n)$ we will define a \emph{product Dirichlet set} $F\star \mathcal P\subset  \Sym^2(\mathbb R^{n+m})$ with the following property:\\

An upper-semicontinuous function $\psi:X \times \mathbb R^m\to \mathbb R\cup \{-\infty\}$ is $F\star \mathcal P$-subharmonic if and only if 
\begin{enumerate}
\item For each $x\in X$ the function $y\mapsto \psi(x,y)$ is convex and
\item For each affine linear $\Gamma:\mathbb R^n\to \mathbb R^m$ the function $x\mapsto \psi(x,\Gamma(x))$ is $F$-subharmonic on $X$.
\end{enumerate}

The parallel with the discussion in the previous section should be clear (in fact what is written there is precisely this construction when $F$ is taken to be $F_{sub}$).   With these definitions we can state our general form of Prekopa's Theorem.

\begin{theorem}[Prekopa's Theorem for F-subharmonicity]
Assume that $F\subset \Sym^2(\mathbb R^n)$ is a convex Dirichlet set and that $\psi:X\times \mathbb R^m \to \mathbb R\cup \{-\infty\}$ is  $F\star \mathcal P$-subharmonic.  Then the function
$$x\mapsto -\log \int_{\mathbb R^m} e^{-\psi(x,y)} dy.$$
is $F$-subharmonic on $X$.
\end{theorem}

In turn this gives the following generalization of the Brunn-Minkowski Theorem.  For this we need to generalize the notion of a convex subset $K$, which we take as the existence of a defining $F\star P$-subharmonic function.  The precise statement is:

\begin{theorem}[Brunn-Minkowski inequality for $F$-subharmonicity]\label{thm:BMI:intro} Let $X\subset \mathbb R^n$ be open.  
Assume that $F\subset \Sym^2(\mathbb R^n)$ is a Dirichlet set that is a convex cone over $0$.  Suppose that $K\subset X\times \mathbb R^m$ is closed and bounded and such that 
\begin{enumerate}
\item $(\Int K)_x = \Int(K_x)$ and is non-empty for all $x\in X$.
\item  There is an $F\star \mathcal P$-subharmonic function $\rho$ on a neighbourhood of $K$ such that $K = \{\rho \le 0\}$.  
\end{enumerate}
Then the map $x\to -\log \vol(K_x)$ is  $F$-subharmonic on $X$.
\end{theorem}

Also from Prekopa's Theorem we can deduce easily a minimum principle.

\begin{theorem}[Minimum Principle for F-subharmonicity]\label{thm:minimumintroduction}
Let $X\subset \mathbb R^n$ be open and assume that $F\subset \Sym^2(\mathbb R^n)$ is a Dirichlet set.  Let $V\subset \mathbb R^m$ be convex and suppose that $\psi$ is $F\star \mathcal P$-subharmonic on $X\times V$.  Then the function 
$$x\mapsto \inf_y \psi(x,y)$$
is $F$-subharmonic.
\end{theorem}

Bringing this back to interpolation, suppose again  $\Omega\subset \mathbb R^n$ is open and bounded with smooth boundary and we are given a family of convex functions $\phi_\tau:\mathbb R^m\to \mathbb R$ such that $(\tau,y)\mapsto \phi_\tau(y)$ is continuous.   We then have an interpolation of this data given by the Perron envelope

\begin{equation}\Phi:= {\sup}^*\left\{\zeta: \begin{array}{l} \zeta \text{ is upper-semicontinuous on }\overline{\Omega}\times \mathbb R^m, \\
F\star \mathcal P\text{-subharmonic on }\Omega\times \mathbb R^m \text{ and } \Phi|_{\partial \Omega\times \mathbb R^m} \le \phi\end{array}\right\}.\label{eq:perron:intro}
\end{equation}


Under suitable boundary convexity of $\Omega$ (that is defined in terms of $F$) and a mild technical assumption on the data $\phi_\tau$ (see Definition \ref{def:locallycomparable})  will see that this envelope is $F\star \mathcal P$-subharmonic and $\Phi|_{\partial \Omega\times \mathbb R^m} = \phi$.  In particular for each $x\in \Omega$ the map $y\mapsto \Phi(x,y)$ is convex, and thus we have an interpolation of the boundary data $\phi$ for each such Dirichlet set $F$ to which our Prekopa Theorem applies.    Moreover we will show in Proposition \ref{prop:harmonicasenvelope} that when $F=F_{sub}$ then $\Phi$ is precisely the harmonic interpolation discussed above.

\subsection{Comparison with other work} 

We refer the reader to Gardner's survey \cite{Gardner} for an account of the classical Brunn-Minkowski Theorem and  Prekopa's Theorem (also called the Brunn-Minkowski inequality and the Prekopa  or Prekopa-Leindler inequality respectively).  The reader will also find there many connections to other inequalities in geometry and analysis.  

  The complex version of the Brunn-Minkowski and Preokopa Theorems is a theme in work of Berndtsson \cite{Berndtsson_Prekopa,Berndtsson_convexityKahler,Berndtsson_survey,Berndtsson_openness,Berndtsson_realandcomplexBM,Berndtsson_complexBM_and_geometry} where it has in particular found applications to positivity of vector bundles and Fano manifolds (see also \cite{Paun_survey} for a survey).    This has since been taken up by others (e.g. Cordero-Erausquin \cite{Cordero-Erausquin} and Nguyen \cite{Nguyen}).

The ideas of F-subharmonicity we use are taken from the works of Harvey-Lawson (e.g. \cite{HL_Potentialalmostcomplex,HL_Dirichletdualitymanifolds,HL_Geometricplurisubharmonicity,HL_Dirichletduality,HL_equivalenceviscosity,HL_pconvexity,HL_Dirichletprescribed}) who in turn are building on various parts of the viscosity approach to differential equations (e.g. Krylov \cite{Krylov}).     The idea of proving the Brunn-Minkowski and Prekopa Theorem's in this context appears to be new, although it does have some overlap with existing ideas (some more of of which are described below).    For instance looking closely one can find overlap between a particular case of our product Dirichlet set and the version of the Prekopa Inequality found in \cite{BrascampLieb} (compare in particular Proposition \ref{prop:characterization1} and  \cite[(4.7)]{BrascampLieb}).

The minimum principle for plurisubharmonic functions goes under the name of the Kiselman Minimum Principle \cite{KiselmanInvent}.  The authors of this paper have previously proved a minimum principle for $F$-subharmonicity \cite{ross_nystrom_minimum} that is slightly stronger that Theorem \ref{thm:minimumintroduction} in that it applies also to subequations $F$ that may depends on the gradient as well as the Hessian part (see also Darvas-Rubinstein \cite{darvas-Rubinstein} for earlier work in this direction).

The authors do not have sufficient expertise to survey the huge amount of work on interpolation, so instead the reader is referred to \cite{BoInterpolation,Coifmanetal,SemmesCoifmanInterpolation,Rochberg_Semmes,Rochberg_interpolation} for just a few examples where this is studied.  In this paper we focus on interpolation of convex functions on finite dimensional vector spaces, but there has been considerable interest in functional analysis on interpolation in infinite dimensions (see \cite{interpolationbook} for an introduction).    Emphasizing in particular the complex setting, Semmes \cite{SemmesInterpolation} describes various aspects of the interpolation problem that are closely related to the work in this paper.

There is also some overlap with the ``generalized subharmonicity" work of Slodkowski \cite{SlodkowskiI,SlodkowskiIII,SlodkowskiII,Slodkowski_harmonic_I,Slodkowski_harmonic_II,Slodkowski_harmonic_III} (who also uses the term ``harmonic interpolation"), although the authors find that casting these ideas within the framework of $F$-subharmonicity is much clearer.    

\subsection{Future Directions}  The theory of $F$-subharmonicity extends to the complex case (where one considers the complex Hessian) as well as to Riemannian manifolds.  We hope to consider whether our Prekopa Theorem holds in these settings in a future work.

\subsection{Organization}  The next section contains a brief account of the pieces from the theory of F-subharmonic functions that we will need.  In \S\ref{sec:products} we define with some care the product of Dirichlet sets, and give some examples.  The smooth case of Prekopa's Theorem for $F$-subharmonicity is then given in \S\ref{sec:prekopasmooth} using a formula for the Hessian due to Ball-Barthe-Naor \cite{BallBartheNaor}.   The general case (so without the smoothness hypothesis) is proved in \S\ref{sec:prekopanonsmooth} using an approximation argument.   Then in \S\ref{sec:BMandminimum} and \S\ref{sec:minimumprinciple} we show how this implies our Brunn-Minkowski Theorem and Minimum Principle for $F$-subharmonicity.     In \S\ref{sec:example} we give a simple but very explicit example of the Brunn-Minkowski theorem in action, and in \S\ref{sec:Bo} we show how our work implies Berndtsson's Prekopa Theorem.  Finally in \S\ref{sec:interpolation} we discuss the interpolation problem for convex functions and convex sets, and prove Theorems \ref{thm:introBM} and \ref{thm:introPrekopa}.

\subsection{Acknowledgements} The authors thank   Bo Berndtsson, Ruadha\'{i} Dervan, Tommy Murphy, Lars Sektnan and Xiaowei Wang for helpful conversations.  We also thank Tristan Collins and Sebastien Picard for sharing with us their provisional work on geodesics in the space of $m$-subharmonic functions and Richard Rochberg for pointing out interesting references on interpolation.    The first author is supported by the NSF grant DMS-1749447.  The second author is supported by a grant from the Swedish Research Council and a grant from the G\"{o}ran Gustafsson Foundation for Natural Science and Medicine.

\section{Background on $F$-subharmonicity}\label{sec:backgroundFsubharmonicity}

\subsection{Basics}

We will need only a small amount of the theory of $F$-subharmonicity most of which can be found in \cite{HL_Dirichletduality}.  Let $\Sym^2(\mathbb R^n)$ denote the set of $n\times n$ symmetric matrices and let $\mathcal P=\mathcal P_n$ denote the subset of positive definite symmetric matrices.
\begin{definition}[Dirichlet sets] A subset $F\subset \Sym^2(\mathbb R^n)$ is called a \emph{Dirichlet set} if it is non-empty, proper, closed and
$$ F+ \mathcal P\subset F.$$
A \emph{convex Dirichlet set} is such an $F$ that is convex as a subset of $\Sym^2(\mathbb R^n)$, and it is a \emph{cone over 0} if $A\in F$ and $t\ge 0$ implies $tA\in F$. 
\end{definition}

\begin{definition}[$F$-subharmonicity]
Let $X\subset \mathbb R^n$ be open, $F\subset \Sym^2(\mathbb R^n)$ be a Dirichlet set, and $f:X\to \mathbb R\cup \{-\infty\}$ be upper semicontinuous.  We call $g$ a \emph{test function} for $f$ at $x_0\in X$ if $g$ is defined on a neighbourhood of $x_0$, is smooth and has $g\ge f$ on this neighbourhood and $g(x_0) = f(x_0)$.  We say that $f$ is \emph{$F$-subharmonic} if for all $x_0\in X$ and test functions $g$ at $x_0$ we have $\Hess_{x_0}(g)\in F$.

The set of $F$-subharmonic functions on $X$ is denoted by $F(X)$.
\end{definition}

Observe that by definition if $f(x_0)=-\infty$ then there can be no test function at $x_0$ so the condition there is vacuous; in particular the function $f\equiv -\infty$ is trivially $F$-subharmonic. Clearly being $F$-subharmonic is a local condition, by which we mean that if $\{X_{\alpha}\}_{\alpha\in \mathcal A}$ is an open cover of $X$ then $f\in F(X)$ if and only if $f\in F(X_{\alpha})$ for all $\alpha$.   It is not hard to show \cite[p17]{HL_Dirichletduality} that if $f$ is smooth on an open $X\subset \mathbb R^n$ then
$$f\in F(X) \text{ if and only if } \Hess_x(f)\in F \text{ for all }x\in X.$$

\begin{remark}[Comparison with other terminology] It is also possible to consider generalized subharmonicity that depends not only on the Hessian but also on the gradient as well as the value of the function at a given point by considering (suitable) subsets $F$ of the space of second order jets.  This is the point of view taken in \cite{HL_Dirichletdualitymanifolds} (and also our previous work \cite{ross_nystrom_minimum}) in which such an $F$ is referred to as a \emph{subequation}.   Certainly Dirichlet sets  are examples of subequations (more precisely in the terminology of \cite{HL_Dirichletdualitymanifolds}  they would be called a ``constant coefficient subequation that depend only on the Hessian part") \end{remark}

The following lists some of the basic properties satisfied by $F$-subharmonic functions.
\begin{proposition}\label{prop:basicproperties}\
\begin{enumerate}
\item (Maximum Property) If $f,g\in F(X)$ then $\max\{f,g\}\in F(X)$.
\item (Decreasing Sequences) If $f_j$ is decreasing sequence of functions in $F(X)$ (so $f_{j+1}\le f_j$ over $X$) then $f:=\lim_j f_j$ is in $F(X)$.
\item (Uniform limits) If $f_j$ is a sequence of functions on $F(X)$ that converge locally uniformly to $f$ then $f\in F(X)$.
\item (Families locally bounded above) Suppose $\mathcal F\subset F(X)$ is a family of $F$-subharmonic functions locally uniformally bounded from above.  Then the upper-semicontinuous regularisation of the supremum
$$ f:= {\sup}^*_{f\in \mathcal F} f$$
is in $F(X)$.
\item (Convexity) If $F$ is convex then $F(X)$ is convex.
\end{enumerate}
\end{proposition}
\begin{proof}
The first four statement are elementary \cite[p16]{HL_Dirichletduality}.  The final one, which is more involved to prove, is \cite[Remark B.9]{HL_Dirichletduality}.
\end{proof}

\subsection{Dirichlet Duality}

\begin{definition} Let $F\subset \Sym^2(\mathbb R^n)$ be a Dirichlet set.  The \emph{Dirichlet dual} of $F$ is defined to be
$$\tilde{F} = \sim(-\Int(F)) = - (\sim \Int(F))$$
where $\sim$ means set-theoretic complement (see \cite[Sec 4]{HL_Dirichletduality}).
\end{definition}

It is easy to check that $\tilde{F}$ is also a Dirichlet set.  The key property we will use of this is the following  \cite[Definition 4.4,Remark 4.9]{HL_Dirichletduality}: if $u\in F(X)$ and $v\in \tilde{F}(X)$ is smooth then $u+v$ is subaffine.

\subsection{Boundary Convexity}
Let $F\subset \Sym^2(\mathbb R^n)$ be a Dirichlet set and $B\in \Sym^2(\mathbb R^n)$. The \emph{ray set with vertex} $B$ \emph{associated to F} is given by 
$$\overrightarrow{F_B} := \{ A\in \Sym^2(\mathbb R^n): \exists \lambda_0 : B+\lambda A\in F \text{ for all } \lambda\ge \lambda_0\}.$$
Although $\overrightarrow{F_B}$ may not be closed, its closure is independent of choice of $B$ and is called the \emph{ray set} associated to $F$ and denoted $\overrightarrow{F}$.  Then $\overrightarrow{F}$ is also a Dirichlet set \cite[\S 5]{HL_Dirichletduality} and if $F$ is a cone over the origin, then $\overrightarrow{F}=F$.\\

Now let $\Omega\subset \mathbb R^n$ be open and smoothly bounded.   A \emph{local defining function} for $\partial \Omega$ near $\tau\in \partial \Omega$ is a smooth function $\rho$ defined on some neighbourhood of $\tau$ such that $\Omega  = \{\rho<0\}$ and $\nabla \rho\neq 0$.   Let $T = T_\tau \partial \Omega$ denote the tangent space to the boundary.

\begin{definition}
We say that $\Omega$ is $\overrightarrow{F}$-convex if for all  $\tau\in \partial \Omega$ we have $\Hess_{\tau} \rho |_{T} = B|_T$ for some $B\in \Int(\overrightarrow{F})$.
\end{definition}

We refer the reader to \cite[\S 5]{HL_Dirichletduality} for more about boundary convexity, in particular how it can be described in terms of the second fundamental form of the boundary and the existence of a global defining function for $\Omega$ \cite[Theorem 5.12]{HL_Dirichletduality}.

\subsection{Some examples of convex Dirichlet Sets}\label{sec:examplessubequations}

\begin{example}[Convexity]
If $F=\mathcal P$ is the set of positive definite matrices then $F$ is a convex Dirichlet set that is a cone over the origin.  Then $F$-subharmonic functions are precisely the convex functions \cite[Prop 4.5]{HL_Dirichletduality}
\end{example}

\begin{example}[Subharmonicity]
Let $F_{sub} = \{A\in \Sym^2(\mathbb R^n): \tr(A)\ge 0\}$.  This is a convex Dirichlet set which is a cone over $0$, and the $F_{sub}$-subharmonic functions are those that are subharmonic in the usual sense.
\end{example}

\begin{example}[Dirichlet sets given by constraints on eigenvalues]
Suppose that $E\subset \mathbb R^n$ is symmetric (i.e. invariant under permuting the factors of $\mathbb R^n$).  Write $\lambda(A)$ for the vector of eigenvalues of a symmetric matrix $A$ (taken in any order) and set
$$F_E: = \{ A\in \Sym^2(\mathbb R^n): \lambda(A)\subset E\}.$$
If one assumes that $E$ is closed and that $E+ \mathbb R_{\ge 0}^n\subset E$ then $F_E$ is a Dirichlet set (see \cite[Sec. 14]{HL_Dirichletdualitymanifolds}).    The reader will observe that this generalizes the previous two examples, and clearly gives much more.
\end{example}

We claim that if $E$ is chosen to also be convex then $F_E$ is convex.    This follows from the fact \cite[Corollary 1]{Davis} that if $g:\mathbb R^n\to \mathbb R$ is a symmetric convex function then the matrix  function $\hat{g}(A) := g(\lambda(A))$ is also convex.   Any convex $E$ can be written as an intersection of half spaces $\{f_i\le 0\}$ where $f_i:\mathbb R^n\to \mathbb R$ is linear.  For any permutation $\sigma\in \Sigma_n$ let $f_i^{\sigma}(x_1,\ldots,x_n) = f(x_{\sigma(1)},\ldots, x_{\sigma(n)})$ and set $g_i = \max_{\sigma\in \Sigma_n} \{ f_i^\sigma\}$,  Then each $g_i$ is convex and $E=\bigcap_i \{g_i\le 0\}$.   So since the corresponding matrix functions $\hat{g_i}$ are also convex, we conclude that $F_E$ is convex.



\section{Products of Dirichlet Sets}\label{sec:products}

\subsection{Definition of products}

In this section we define product of Dirichlet sets.   To do so fix positive integers $n,m$.  Then any $A\in\Sym^2(\mathbb R^{n+m})$ can be written in block form as 
\begin{equation} A = \left(\begin{array}{cc} B & C \\ C^T & D \end{array}\right)\label{eq:blockform}\end{equation}
where $B\in \Sym^2(\mathbb R^n), D\in \Sym^2(\mathbb R^m)$ and $C\in M_{n\times m}$.  We set $\pi(A):=D$ and for $\Gamma\in  M_{m\times n}$ 
\begin{equation}i_{\Gamma}^* A := (I, \Gamma)^T A (I,\Gamma) = B + C\Gamma + \Gamma^T C^T + \Gamma D \Gamma^T.\label{eq:defistargamma}\end{equation}
\begin{definition}
For $F\subset \Sym^2(\mathbb R^n)$ and $G\subset \Sym^2(\mathbb R^m)$ define
$$ F\star G := \left\{  A\in \Sym^2(\mathbb R^{n+m}) : \pi(A)\in G \text{ and } i_{\Gamma}^* A \in F \text{ for all } \Gamma\in M_{m\times n}\right\}.$$
\end{definition}


It is easy to verify that if $F$ and $G$ are Dirichlet sets then so is $F\star G$. We will be most interested in the case that $G= \mathcal P$ and $F$ is convex in which case one of the conditions that define this product can be dropped.

\begin{lemma}
Assume $F$ is a convex Dirichlet set.  If $i_{\Gamma}^* A \in $ for all $\Gamma\in M_{m\times n}$ then $\pi(A)\in \mathcal P$.   In particular
$$ F\star \mathcal P := \left\{  A\in \Sym^2(\mathbb R^{n+m}) :  i_{\Gamma}^* A \in F \text{ for all } \Gamma\in M_{m\times n}\right\}.$$
\end{lemma}
\begin{proof}
We will show that if the first statement does not hold then actually $F$ is not proper (which is forbidden in the definition of a Dirichlet set).    So suppose there is an $A\in F\star \mathcal P$ written in block form as in \eqref{eq:blockform} but with $D\notin \mathcal P$, i.e. there is a $w\in \mathbb R^m$ with $w^T D w=-1$.  Next fix linear maps $\Gamma_i:\mathbb R^n\to span(w)\subset \mathbb R^m$ for $i=1,\ldots, N$ taken so that $\bigcap_i ker(\Gamma_i)= \{0\}$.   We then have $B + C\Gamma_i + \Gamma_i^T C + \Gamma_i D \Gamma_i^T \in F$ for each of these $\Gamma_i$.  So by convexity of $F$
$$ E : = B + \frac{1}{N}\sum_i C\Gamma_i + \Gamma_i^T C^T + \Gamma_i^T D\Gamma_i\in F.$$
Now consider any $B'\in \Sym^2(\mathbb R^n)$.  We claim that by rescaling each $\Gamma_i$ if necessary, that $B'-E\in \mathcal P$.  From this it follows that $B' = E+B'-E\in F + \mathcal P \subset F$ which means $F = \Sym^2(\mathbb R^n)$ which is absurd

For the claim suppose $v\in \mathbb R^n$ is a unit vector.   Then $\Gamma_i(v) = \lambda_i w$ for some $\lambda_i\in \mathbb R$ with at least one of the $\lambda_i$ non-zero. Then $B'-E = \sum_i \Gamma_i^T D \Gamma_i + \Delta$ where $\Delta$ is a sum of terms that are constant or linear in $\Gamma_i$.  Then putting $\delta: = v^T \Delta v$ we have
$$ v^T(B'-E)v =  -\sum_i v^T \Gamma_i^T D \Gamma_i v + \delta = -\sum_i \lambda_i^2w^T D w + \delta  = \sum_i \lambda_i^2 + \delta.$$
Now, replacing each $\Gamma_i$ by $t\Gamma_i$ for some large $t$ replaces each $\lambda_i$ with $t\lambda_i$ and so the (strictly positive) quadratic term dominates the term $\delta$ which is $O(t)$, making $v^T (B'-E)v$ strictly positive. Moreover we can do this uniformly over all unit vectors $v$.   So the claim is shown and the proof is complete.
\end{proof}

The following characterizes $F\star \mathcal P$-subharmonic functions (and confirms that the definition from this section agrees with that from the introduction). 

\begin{proposition}\label{prop:slices}
Assume that $F\subset \Sym^2(\mathbb R^n)$ and $G\subset \Sym^2(\mathbb R^m)$ are Dirichlet sets.   Let $X\subset \mathbb R^{n}$ and $Y\subset \mathbb R^m$ be open and $f:X\times Y\to \mathbb R\cup \{-\infty\}$ be upper-semicontinuous.  The following are equivalent
\begin{enumerate}
\item $f$ is $F\star G$-subharmonic
\item $\begin{array}{ll}
x\mapsto f(x,y_0+ \Gamma x) \text{ is } F\text{-subharmonic for all } y_0\in Y\text{ and }\Gamma\in M_{m\times n} \text{ and} \\
$$y\mapsto f(x_0,y) \text{ is } G\text{-subharmonic for all } x_0\in X.\\
\end{array}$
\end{enumerate}
\end{proposition}
\begin{proof}
See \cite[Proposition 3.12]{ross_nystrom_minimum}.
\end{proof}

\begin{remark}[Products of Subequations that depend also on the gradient part]\label{rmk:productswithgradient}
It is possible also to consider products of subequations that depend on the entire jet rather than just the Hessian part (the authors previously considered this in \cite{ross_nystrom_minimum} where we use the notation $F\#G$ instead of $F\star G$.). 

To describe this, given $$p = \begin{pmatrix} p_1 \\ p_2 \end{pmatrix} \in \mathbb R^{m+n}$$ and for $\Gamma \in M_{m\times n}$ set
$$i_{\Gamma}^* p = p_1 + \Gamma^T p_2.$$
Then for $F\subset J^2_n = \mathbb R \oplus \mathbb R^n \oplus \Sym^2(\mathbb R^n)$ we defined
$$F\# \mathcal P = \{ (r,p,A) \in \mathbb R\times \mathbb R^{n} \times \Sym^2(\mathbb R^{n+m}): \pi(A) \in \mathcal P \text{ and }(r,i_{\Gamma}^* p , i_{\Gamma}^* A) \in F \text{ for all } \Gamma\}.$$
\end{remark}


\subsection{Characterization of Products}

When $G =\mathcal P$ and $F$ is convex we have a more concrete characterization of $F\star P$.   To discuss this, observe first that any affine subspace in $\Sym^2(\mathbb R^n)$ is given by the kernel of a function of the form $\phi(B) = \tr(UB)-c$ for some non-zero $U\in \Sym^2(\mathbb R^n)$ and $c\in \mathbb R$.     We say that an affine function $\phi:\Sym^2(\mathbb R^n)\to \mathbb R$ is a \emph{supporting hyperplane} of $F$ if $\phi\ge 0$ on $F$ and there is some point $A\in F$ with $\phi(A)=0$. 

\begin{definition}
Let $F\subset \Sym^2(\mathbb R^n)$ be a closed and convex.   The \emph{null space} of $F$ denoted by either $\nul(F)$ or $\nul(F)_{m\times n}$ is defined to be
 $$\nul(F):= \{ C\in M_{m\times n}: UC =0 \,\, \forall \text{ supporting hyperplanes } B\mapsto \tr(UB) -c \text{ of }F \}.$$
 \end{definition}
 
In many cases $\nul(F)$ is trivial, for instance we have the following. 

\begin{lemma}
Suppose that $F\subset \Sym^2(\mathbb R^n)$ is a convex cone Dirichlet set and that $\{ B\in F: \tr(B)=1\}$ is compact.  Then $\nul(F)=\{0\}$. 
\end{lemma}
\begin{proof}
The statement is trivial when $n=1$, so assume $n\ge 2$.  Since $F$ is a cone over $0$ it can be written as an intersection of subspaces 
$$F  = \bigcap_{U\in \mathcal U} \{ B : tr(B U)\ge 0\}.$$
Suppose for contradiction there is a non-zero $C\in \nul(F)$, which by rescaling we may assume $\tr(CC^T)=1$.  Then for any $U\in \mathcal U$ we have $UC=0$.  Fix a  $E$ with $CE\neq 0$ but $\tr(CE)=0$ and consider $B = \lambda CE  + CC^T$ for $\lambda\in \mathbb R$.    Then $UB=0$ so $B\in F$ and moreover $\tr(B)=1$.  But this contradicts compactness of $\{ B\in F : \tr(B)=1\}$.
\end{proof}

\begin{example} If $F=\mathcal P$ or if $F = F_{sub}$ then $\nul(F) = \{0\}$.  However there are examples where $\nul(F)$ is non-trivial (for instance consider the Dirichlet set consisting of symmetric matrices in block form $\left(\begin{array}{cc} B & C \\ C^t & D \end{array}\right)$ with $\tr(B)\ge 0$)
\end{example}

\begin{proposition}\label{prop:characterization1}
Suppose $F\subset \Sym^2(\mathbb R^n)$ is a convex Dirichlet set.  Consider the block matrix
$$ A := \left(\begin{array}{cc} B & C \\ C^t & D \end{array}\right)\in \Sym^2(\mathbb R^{n+1})$$
with $D\in \mathbb R$.  Then $A\in F\star \mathcal P$ if and only if either
\begin{enumerate}
\item $D> 0$ and $B - C D^{-1} C^t \in F$ or
\item $D=0$ and $C\in \nul(F)$ and $B \in F$.
\end{enumerate}
\end{proposition}

The proof will be by reducing to the case that $F$ is a half-space.

\begin{lemma}\label{lemma:quadraticnonsense}
Let $U\in \mathcal P$ be non-zero, let $c\in \mathbb R$ and set
$$\phi(B) = \tr(UB) \text{ for } B\in \Sym^2(\mathbb R^n)$$
Fix $B\in \Sym^2(\mathbb R^n)$, $C\in \mathbb R^n$ and $D\in \mathbb R$.  Then
\begin{equation}\label{eq:lowerboundphi} \phi( B + C \Gamma + \Gamma^T C + \Gamma^T D \Gamma) \ge c\text{ for all } \Gamma\in \mathbb R^n\end{equation}
if and only if either
\begin{enumerate}
\item $D>0$ and $\phi(B - C D^{-1} C^T)\ge c$ or
\item $D=0$ and $UC=0$ and $\phi(B)\ge c.$
\end{enumerate}
\end{lemma}
\begin{proof} 
Since $U\in \mathcal P$ is non-zero we can write $U = V^T V$ for some non-zero  $V\in \mathcal P$.  We will use throughout that if $B\in \mathcal P$ then $\phi(B) = \tr(UB) \ge 0$.

Assume now that conditions (1) and (2) hold and set
$$f(\Gamma): =  \phi( B + C \Gamma + \Gamma^T C^T + \Gamma^t D \Gamma).$$

\noindent {\bf Case 1: $D>0$ } Suppose first that $V$ is such that $\phi(\Gamma^T D \Gamma)=0$.  Using that $D>0$ this implies  $0=\tr( V^TV \Gamma^T\Gamma)=\tr( (V\Gamma^T)^T V\Gamma^T)$ and so $V\Gamma^T=0$.  Thus we in fact have
$$\phi(C\Gamma + \Gamma^T C^T) = \tr( V^TV (C\Gamma + \Gamma^T C^T)) =\tr (V C\Gamma V^T) + \tr(V \Gamma^T C^T V^T) =0.$$
So
$$f(\Gamma) = \phi(B) \ge \phi(B - C D^{-1} C^T) \ge c$$
where the last inequality uses (1).

Suppose next that $V$ is such that $\phi(\Gamma^T D \Gamma)\neq 0$.  Then in fact  $\phi(\Gamma^T D \Gamma)>0$ since $\Gamma^T D \Gamma\in \mathcal P$.   Thus $f(\Gamma)\to \infty$ as $\|\Gamma\|\to \infty$.  So the minimum value of $f(\Gamma)$ is attained at a critical point.  Any such critical point $\Gamma_0$ satisfies for all $\Delta$ that
 \begin{equation} 0 = Df|_{\Gamma_0}(\Delta) = \phi(C \Delta+ \Delta^T C^T + \Delta^T D \Gamma_0 + \Gamma_0^T D\Delta).\label{eq:derivativevanishes}\end{equation}
 In particular this will hold when $\Delta = D^{-1}C^T + \Gamma_0$.  After some manipulation and cancellation this implies
 
 $$ 0 = \phi( C D^{-1} C^T  + C\Gamma_0 + \Gamma_0^t C^T + \Gamma_0^TD \Gamma_0)$$
 which in turn gives
 $$f(\Gamma_0) =  \tr( B + C\Gamma_0 + \Gamma_0^t C^T + \Gamma_0^T D \Gamma_0) = tr(B - C D^{-1} C^T) \ge c$$
 where the last inequality uses (1).  Thus we have $f(\Gamma)\ge c$ for all $\Gamma$.  \medskip
 
\noindent {\bf  Case 2:  $D=0$ }  Again fix $\Gamma$ and consider
$$f(\Gamma):= \phi( B +C \Gamma + \Gamma^T C^T + \Gamma^t D \Gamma) = \phi( B ) + \phi(C \Gamma + \Gamma^t C^T).$$
Since we are assuming (2) holds we have $V^T V C=UC=0$.  Thus
$$ \phi(C \Gamma + \Gamma^T C^T) = tr(V^T V(C \Gamma + \Gamma^T C^T) = \tr( V^TV C \Gamma) + \tr(\Gamma^TC^T V^T V)=0.$$  Thus in fact
$$f(\Gamma) = \phi(B)\ge c$$
where the last inequality uses the last part of (2).  \medskip

For the converse suppose that \eqref{eq:lowerboundphi} holds.  If $D>0$ then this in particular holds for $\Gamma_0 = -D^{-1}C^T$ giving $\phi(B - C^T D^{-1} C)= f(\Gamma_0) \ge c$.  Suppose then $D=0$.  Applying this with $\Gamma=0$ gives $\phi(B)\ge 0$.  On the other hand for any $\Gamma$ the fact that $D=0$ means we have $\phi(B + C\Gamma + \Gamma^t C^t)=0$ and since $\Gamma$ can be scaled arbitrarily this implies $\phi(C\Gamma + \Gamma^T C)=0$.  In particular this applies when $\Gamma = C^T$ and so $0 = \phi( CC^T) = \phi( V^T V C C^T)$ and so $VC=0$ and thus $UC = V^TVC=0$ as well.  
\end{proof}

\begin{lemma}\label{lem:halfspaces}
Let $F\subset \Sym^2(\mathbb R^n)$ is closed convex and $F+\mathcal P\subset F$.  Then $F$ can be written as an intersection of half spaces given by supporting hyperplanes of the form
$$ \{ B : \tr(VV^TB)\ge c\}$$
where each $V\in \mathcal P$.
\end{lemma}

\begin{proof}
As $F$ is closed and convex we know that it can be written as an intersection of half spaces given by supporting hyperplanes, each of which can be written as $$ \{ B : \tr(UB)\ge c\}$$ for some $U\in \Sym^2(\mathbb R^n)$.

Now we must have $\tr(UB)\ge 0$ for all $B\in \mathcal P$.  For if $B_0\in F$ is fixed and $B\in \mathcal P$ then $B_0 + \lambda B\in F$ for all $\lambda \ge 0$ so $tr(U B_0) +\lambda \tr(UB)\ge c$ for all $\lambda \ge 0$ which implies $\tr(UB)\ge 0$.   Thus $U$ lies in the dual cone to $\mathcal P$, and $\mathcal P$ is self-dual, so in fact we have $U\in \mathcal P$. Thus we can write $U= V^TV$ for some $V\in \mathcal P$ and the proof is complete.
\end{proof}

\begin{proof}[Proof of Proposition \ref{prop:characterization1}]
 By Lemma \ref{lem:halfspaces} we can write $F$ as an intersection of half spaces of the form
$$ F_{\alpha}  = \{ B : \tr(U_\alpha B) \ge c_{\alpha}\}$$
for some non-zero $U_\alpha\in \mathcal P$ and $c_{\alpha}\in \mathbb R$.  It is easy to check then that $F\star \mathcal P = \cap_{\alpha}( F_{\alpha} \star P)$ and that $\nul(F) = \cap_{\alpha}\nul(F_{\alpha})$.  Thus we may assume from now on that $F$ is given by
$$ F = \{ B : \tr(UB)\ge c\}$$
where $U\in \mathcal P$ is non-zero and $c\in \mathbb R$.  And this case is covered precisely by Lemma \ref{lemma:quadraticnonsense}.
\end{proof}

For higher $m$ we get a similar statement.  

\begin{proposition}\label{prop:characterization2}
Suppose $F\subset\Sym^2(\mathbb R^n)$ is closed and convex and $F+\mathcal P\subset F$.  Consider the block matrix
$$ A := \left(\begin{array}{cc} B & C \\ C^T & D \end{array}\right)\in \Sym^2(\mathbb R^{n+m})$$
Then $A\in F\star \mathcal P$ if and only if
$$D\in P \text{ and }  C(I-D^+D) \in\nul(F)\text{ and }B - C D^{+} C^t \in F$$
where $D^+$ denotes a pseudo-inverse of $D$.
\end{proposition}
\begin{proof}
The proof is essentially the same as the above and we sketch the essential changes.    In fact a look at the proof of Proposition \ref{prop:characterization1} shows that to prove this case it is sufficient to show that Lemma \eqref{lemma:quadraticnonsense} continues to hold with conditions (1) and (2) of the statement replaced with the condition that $D\in \mathcal P$ and $UC(I-D^+D)=0$ and $\phi(B - C D^{+} C^t )\ge c$.  

Assuming that condition holds, and with the notation as in the start of the proof of that Lemma,  we consider two cases depending on the vanishing of $\phi(\Gamma^T D \Gamma)$.  If  $\phi(\Gamma^T D \Gamma)>0$ then exactly the same argument holds shows that $f$ is minimized at a critical point $\Gamma_0$ that must satisfy \eqref{eq:derivativevanishes} for all $\Delta$ and now putting $\Delta = D^+ C^T + \Gamma_0$ gives $0 = \phi(CD^+C^T + C(I + D^+D)\Gamma_0 + \Gamma_0^t(I + DD^+)C^T + 2\Gamma_0^T D\Gamma_0)$.  Using the assumed condition this becomes $\phi(C D^+ C^T + C\Gamma_0 + \Gamma_0^T C^T + \Gamma_0^T D\Gamma_0)=0$ and then after some manipulation $f(\Gamma_0) = \phi(B - CD^{+} C^T)\ge c$.     

If  $\phi(\Gamma^t D \Gamma)=0$ then we write $D = E \begin{pmatrix} \Sigma_r & 0 \\ 0 & 0 \end{pmatrix} E^T$ where $\Sigma_r$ is strictly positive diagonal and $E$ is orthogonal, so $D^+ D = E \begin{pmatrix} I & 0 \\ 0& 0 \end{pmatrix} E^t$.  Writing $E^T \Gamma = \begin{pmatrix} R \\ S \end{pmatrix}$ the condition $\phi(\Gamma^T D \Gamma)=0$ implies, after some manipulation, that $RV^t=0$.   From this another manipulation gives $\phi(C\Gamma + \Gamma^t C)=0$ and so $f(\Gamma) = \phi(B)\ge c$.

\end{proof}

\section{Prekopa Theorem: The Smooth Case}\label{sec:prekopasmooth}

Let $X\subset\mathbb R^n$ be open and suppose $\Omega\subset X\times \mathbb R^m$ is a Borel subset.    For $x\in X$ we let $$\Omega_x := \{ y\in \mathbb R^m : (x,y)\in \Omega\}.$$  
Suppose now $\psi:\Omega\to [-\infty,\infty]$ is upper-semicontinuous and such that for each $x$ the integral $\int_{\Omega_x} e^{-\psi(x,y)} dy$ exists (possibly infinite).   In the case we have in mind, for each $x$ the map $y\mapsto \psi(x,y)$ is convex and thus continuous, so this last hypothesis certainly holds.

\begin{definition}
The \emph{marginal} of $\psi$ is defined to be
$$M(\psi)(x):= M_{\Omega}(\psi)(x) := - \log \int_{\Omega_x} e^{-\psi(x,y)} dy \text{ for } x\in X.$$
\end{definition}

Now suppose $X\subset \mathbb R^n$ is open and ${\psi}(x,y):X\times \mathbb R\to \mathbb R$ be sufficiently smooth and such that for each $x\in X$ the function $e^{-\psi(x,y)}$ is integrable over $\mathbb R$.  To ease notation let 
$ \phi : =M(\psi)$ be the marginal function so
\begin{equation}
e^{-\phi(x)} = \int_{\mathbb R} e^{-{\psi}(x,y)} dy.\label{eq:defofphi}
\end{equation}
We wish to compute the first two derivatives of $\phi$.   The first derivative is easy since, under suitable regularity assumptions, we may differentiate under the integral sign.    Getting a useful formula for the second derivative is more challenging, and the following calculation is taken from  \cite{BallBartheNaor} (but is only considered there in the case that $X$ is a subset of $\mathbb R$, i.e. when $n=1$).  By means of terminology we say that a function on $u(x,y)$ on $X\times \mathbb R^m$ with values in some normed vector space has \emph{polynomial growth in $y$} if there is an integer $q$ such that
 $$ \| u(x,y)\|  = O(\|y\|^q) \text{ uniformly over } X.$$
 

\begin{proposition}\label{prop:Hessiancalc}
Assume $\psi:X\times\mathbb R\to \mathbb R$ is $\mathcal C^2$ and that 
\begin{enumerate}
\item There is a constant $c>0$ such that $\partial_y \psi\ge c$ for $|y|$ sufficiently large (uniformly over $X$) and
\item $d_x\psi$ and $\Hess(\psi)$ have polynomial growth in $y$.
\end{enumerate}
Define
 \begin{align}
\Gamma(x,y) := -\frac{1}{e^{-{\psi}}} \left( d \phi(x) \int_{-\infty}^y e^{-{\psi(x,v) }}dv - \int_{-\infty}^y e^{-{\psi(x,v)}} d_x {\psi} (x,v) dv  \right).\label{def:Gamma}
\end{align}
Then 
\begin{enumerate}
\item $\Gamma$ and $\partial_y \Gamma$ have polynomial growth in $y$.
\item $\phi$ is $\mathcal C^2$,
\begin{equation}\label{eq:derivativephi}d \phi = \frac{\int_{\mathbb R} e^{-\psi} (d_x {\psi}) dy}{\int_{\mathbb R} e^{-\psi} dy}\end{equation}
and
\begin{equation}\label{eq:Hessofpsi}
\Hess(\phi) = \frac{1}{\int_{\mathbb R} e^{-{\psi}} dy} \left( \int_{\mathbb R}  e^{-{\psi}} (\partial_y \Gamma)^T(\partial_y \Gamma)  dy + \int_{\mathbb R}  e^{-{\psi}}  i_{\Gamma}^*\Hess(\psi) dy\right)
\end{equation}
where 
$$i_\Gamma^* \Hess(\psi)= \psi_{xx} +  \psi_{xy}\Gamma  + \Gamma^T \psi_{xy}^T+ \Gamma \psi_{yy} \Gamma^T$$
(so $i_{\Gamma}^*$ is precisely as defined in \eqref{eq:defistargamma}).
\end{enumerate}
\end{proposition}
\begin{proof}
The statement is local in $X$, and so in the following we will allow $X$ to be shrunk without future comment.   The first thing to note is that the conditions we are placing in $\psi$ are certainly enough to differentiate under the integral sign in \eqref{eq:defofphi} so the first derivative of $\phi$ is given by \eqref{eq:derivativephi}.   Using this gives
\begin{align}\Gamma(x,y) &=-\frac{1}{e^{-{\psi}}} \left( d \phi(x) \int_{-\infty}^y e^{-{\psi(x,v) }}dv - \int_{-\infty}^y e^{-{\psi(x,v)}} d_x {\psi} (x,v) dv  \right) \label{eq:Gamma1}\\
&= \frac{1}{e^{-\psi}} \left( d \phi(x) \int_{y}^\infty e^{-{\psi(x,v) }}dv - \int_{y}^\infty e^{-{\psi(x,v)}} d_x {\psi} (x,v) dv \right)  \label{eq:Gamma2}.\end{align}
From this it is easy to check that $\Gamma$ has polynomial growth in $y$.  For instance to control the second integral in \eqref{eq:Gamma2} note that for $v>y\gg 0$ our hypotheses on $\psi$ tell us that $\psi(x,v) \ge \psi(x,y) + c(v-y)$ and $d_x\psi(x,v) = O(y^q)$ so  
$$\frac{1}{e^{-\psi(x,y)}} \int_{y}^\infty e^{-\psi(x,v)} d_x\psi(x,v) dv \le C \int_{y}^{\infty} e^{-c(v-y)} v^q dv = O(y^q)$$
where the last equality is elementary integration by parts.  The first integral in \eqref{eq:Gamma2} is similarly controlled, and to control $\Gamma$ for $y\ll 0$ one argues the same way using \eqref{eq:Gamma1}.  Thus $\Gamma$ has polynomial growth in $y$.  Computing $\partial_y \Gamma$, a similar argument shows that  $\partial_y \Gamma$ also has polynomial growth in $y$.\\


 Now fix a reference point $x_0\in X$.  We define the so-called ``transportation function" $T(x,y)$ by requiring that
\begin{equation} \frac{1}{e^{-\phi(x)}}\int_{-\infty}^{T(x,y)} e^{-\psi(x,u)} du =  \frac{1}{e^{-\phi(x_0)}} \int_{-\infty}^y e^{-\psi(x_0,u)} du.\label{eq:defofT}\end{equation}
The fact that such a $T$ exists follows from the fact that for each fixed $x$ the function $\frac{e^{-\psi(x,y)}}{e^{-\phi(x)}}$ defines a density on $\mathbb R$ with total mass 1 (this is by the definition of $\phi$).   Observe from its definition we clearly have
$$ T(x_0,y)=y,$$
and thus
\begin{equation}\partial_y T(x_0,y) =1.\label{eq:partialyTatx0}\end{equation}

We will need some specifics about the regularity of $T$.    Set
$$\Phi(x,y,z) = \frac{1}{e^{-\phi(x)}}\int_{-\infty}^{z} e^{-\psi(x,u)} du -  \frac{1}{e^{-\phi(x_0)}} \int_{-\infty}^y e^{-\psi(x_0,u)} du$$
so $T$ is defined implicitly by
$$\Phi(x,y,T(x,y)) = 0.$$
Since $\partial_z \Phi = \frac{e^{-\psi(x,z)}}{e^{-\phi(x)}}$ is non-vanishing, the implicit function theorem gives that $T$ is smooth and
\begin{equation} d_x T(x,y) = - \frac{d_x \Phi(x,y,T(x,y))}{\partial_z \Phi(x,y,T(x,y))}.\label{eq:dxT}\end{equation}
Now
$$ d_x\Phi = \frac{1}{e^{-\phi(x)}} \left( d\phi_x \int_{-\infty}^z e^{-\psi(x,u)} du - \int_{-\infty}^z e^{-\psi(x,u)} d_x\psi(x,u) du\right)$$
so evaluating at $x_0,y$ gives 
\begin{align*}d_xT(x_0,y)  &= -\frac{1}{e^{-{\psi(x_0,y)}}} \left( d \phi_{x_0} \int_{-\infty}^y e^{-{\psi(x_0,v) }}dv - \int_{-\infty}^y e^{-{\psi(x,v)}} d_x {\psi} (x_0,v) dv  \right) \\
&= \Gamma(x_0,y)\end{align*}
where $\Gamma$ is as defined in \eqref{def:Gamma}.   Moreover by differentiating \eqref{eq:dxT} and arguing as we did for $\Gamma$ one can check that $\partial_{x_i x_j} T(x_0,y)$ has polynomial growth in $y$ (this is left as an exercise to the reader). \\

We are now ready to compute the Hessian of $\phi$.  First differentiate  \eqref{eq:defofT} with respect to $y$ to get
\begin{equation}   \frac{1}{e^{-\phi(x)}} e^{-\psi(x,T(x,y))} \partial_y T(x,y) =  \frac{1}{e^{-\phi(x_0)}} e^{-\psi(x_0,y)}.\label{eq:functionalequationT}\end{equation}

Then taking the logarithm of \eqref{eq:functionalequationT}  and differentating with respect to $x_i$ gives
$$\partial_{x_i} \phi(x)-\partial_{x_i}\psi(x,T(x,y))  - \partial_y \psi(x,T(x,y)) \partial_{x_i}T(x,y) + \frac{\partial^2_{x_iy} T(x,y)}{\partial_y T(x,y)} = 0,$$
and then differentiating with respect to $x_j$ yields
\begin{align*}0=\partial^2_{x_i x_j} \phi - \partial^2_{x_ix_j}\psi  - \partial^2_{x_iy} \psi \partial_{x_j}T - \partial^2_{x_jy} \psi \partial_{x_i} T - \partial^2_y \psi \partial_{x_j}T \partial_{x_i} T - \partial_y \psi \partial^2_{x_jx_i}T \\
+ {(\partial_yT)^{-2}}(\partial_y T \partial^3_{x_jx_iy}T - \partial ^2_{x_iy}T \partial^2_{x_jy}T)\end{align*}
where this is all computed at the point $(x,T(x,y))$.   Evaluating at $x_0$ and using that $T(x_0,y)=y$ and $\partial_y T(x_0,y)=1$ and $\Gamma(x_0,y) = d_xT(x_0,y)$ this becomes 
\begin{equation}\partial^2_{x_i x_j} \phi = \partial^2_{x_ix_j}\psi  +  \partial^2_{x_iy} \psi \Gamma_j +\partial^2_{x_jy} \psi \Gamma_i+ \partial^2_y \psi \Gamma_j\Gamma_i + \partial_y \psi \partial^2_{x_jx_i}T - \partial_y T \partial^3_{x_jx_iy}T +\partial_{y}\Gamma_i \partial_{y}\Gamma_j\label{eq:Hessphiatxy}\end{equation}
all evaluated at $(x_0,y).$

The first four terms of the right hand side of this add up to the $(i,j)$ entry of $i^*_{\Gamma}\Hess(\psi)$ and the last term is the $(i,j)$ entry of $\partial_y \Gamma^T \partial_y \Gamma$.  So, for all $y$
\begin{equation}\Hess(\phi)({x_0}) = i_{\Gamma}^* \Hess(\psi) + \partial_y \Gamma^T \partial_y \Gamma + \Delta \label{eq:equationforHessphiatanypoint}\end{equation}
where $$\Delta(y)_{ij} := \partial_y\psi \partial^2_{x_jx_i} T |_{x=x_0}-\partial^3_{x_jx_iy} T |_{x=x_0}.$$
The plan is to multiply \eqref{eq:equationforHessphiatanypoint} by $e^{-\psi}$ and integrate with respect to $y$.  To this end note
$$ e^{-\psi(x_0,y)} \Delta(y) = \partial_y (e^{-\psi(x_0,y)} \partial^2_{x_ix_j}T(x_0,y))$$
and so $$ \int_{\mathbb R} e^{-\psi(x_0,y)} \Delta(y) dy =\left[e^{-\psi(x_0,y)}\partial^2_{x_ix_j} T(x_0,y)\right]_{y=-\infty}^{y=+\infty}=0$$ (this last equality uses that $\partial_{x_ix_j}T(x_0,y)$ has polynomial growth in $y$).  Note also that $\Gamma$ and $\partial_y \Gamma$ having polynomial growth in $y$ along with our assumptions on $\psi$ imply that the integrals on the right-hand side of \eqref{eq:Hessofpsi} are finite.   Then as \eqref{eq:equationforHessphiatanypoint} holds for any $(x_0,y)$ we may multiple by both sides by $e^{-\psi(x_0,y)}$ and then integrate over $y$ to give the result we want \eqref{eq:Hessofpsi} at the point $x_0$.\medskip
\end{proof}


\begin{theorem}[Prekopa's Theorem for Smooth Functions]\label{thm:prekopasmooth}
Let $F\subset \Sym^2(\mathbb R^n)$ be a convex Dirichlet set  and $X\subset \mathbb R^n$ be open.  Suppose that $\psi(x,y):X\times \mathbb R^m\to \mathbb R$ is $\mathcal C^2$ and $F\star \mathcal P$-subharmonic and satisfies 
\begin{enumerate}
\item $\partial_{y_i} \psi \ge c>0$ for all $i$ uniformly over $X$ and 
\item $d\psi$ and $\Hess(\psi)$ have polynomial growth in $y$.
\end{enumerate}
Then $M_{X\times \mathbb R^m}(\psi)$ is $F$-subharmonic.
\end{theorem}

\begin{proof}
To ease notation set $\phi:= M_{X\times \mathbb R^m}(\psi)$.\\

Consider first the case $m=1$.  Then the conditions (1) and (2) are precisely those that allow us to apply Proposition \ref{prop:Hessiancalc}.  The assumption that ${\psi}$ is $F\star \mathcal P$-subharmonic says that for any $\Gamma$ and any $(x,y)\in X\times \mathbb R$ we have
$$i_{\Gamma}^* \Hess(\psi)(x,y) \in F$$
In particular this holds for $\Gamma(x,y)$ defined in \eqref{def:Gamma} and thus the quantity
$$ \frac{1}{\int_{\mathbb R} e^{-{\psi}} dy} \int_{\mathbb R} e^{-{\psi}}  i_{\Gamma}^* \Hess({\psi})dy$$
is an average of elements of $F$.  Hence as $F$ is convex and closed this also lies in $F$.

On the other hand the quantity 
$$\frac{1}{\int_{\mathbb R} e^{-{\psi}} dy}  \int_{\mathbb R} e^{-{\psi}} (\partial_y \Gamma)^t(\partial_y \Gamma)  dy $$
is clearly positive semidefinite, and thus from \eqref{eq:Hessofpsi} we deduce that $ \Hess(\phi)_x\in F$.
Thus $\phi$ is $F$-subharmonic, completing the proof when $m=1$.\\

The statement for higher values of $m$ is easily proved by induction since we can integrate in the $y$ direction one variable at a time.   The reason that this work is that if $\mathcal P_m$ denotes the subset of semipositive matrices in $\Sym^2(\mathbb R^m)$ then by associativity of products \cite[Appendix B]{ross_nystrom_minimum} $$F\star \mathcal P_{m} = F \star (\mathcal P_{m-1}\star \mathcal P_1) = (F\star \mathcal P_{m-1}) \star \mathcal P_1.$$  Then we write $\tilde{X} := X\times \mathbb R^{m-1}$ and for $z = (x,y_1,\ldots,y_{m-1})\in \tilde{X}$ and $y_m\in \mathbb R$ set $$\tilde{\psi}(z,y_m) = \psi(x,y).$$

Our hypothesis on $\psi$ tell us that $\tilde{\psi}$ satisfies the regularity conditions of Proposition \ref{prop:Hessiancalc} and is $(F\star \mathcal P_{m-1}) \star \mathcal P_1$-subharmonic.   Thus if we set
$$\zeta: = M_{\tilde{X}\times \mathbb R}(\tilde{\psi})$$
then $\zeta$ is $F\star \mathcal P_{m-1}$-subharmonic on $\tilde{X}$.  Moreover $\zeta$ is $\mathcal C^2$, and the formulae for its derivatives given by Proposition \ref{prop:Hessiancalc} tell us we may apply the inductive hypothesis to $\zeta$.  Thus $M_{X\times \mathbb R^{m-1}}(\zeta)$ is $F$-subharmonic.    But this marginal function is obtained from $\psi$ by first integrating over the $y_m$ variable and then over the remaining $y_1,\ldots,y_{m-1}$, i.e.
$$ M_{X\times \mathbb R^m}(\psi) = M_{X\times \mathbb R^{m-1}}(\zeta)$$
which proves the induction step.
\end{proof}

\begin{remark}[Subequations that depend on the gradient part] There is a condition on $F$ we can make that allows us to extend the above Prekopa's Theorem to subequations that depends also on the gradient part.   Let $X\subset \mathbb R^n$ be open and suppose $F\subset J^2(X)$ is a convex primitive subequation, independent of the constant part and also that
\begin{equation} (p,A)\in F \Rightarrow (p +q, A + q q^T) \in F \text{ for all } q\in \mathbb R^n.\label{eq:extrasssumptiononF}\end{equation}
Assume now $\psi$ is as in the statement of Theorem \ref{thm:prekopasmooth} and we claim that $M(\psi)$ is $F$-subharmonic.  

It is sufficient to prove this when $m=1$ (for then the proof of higher $m$ follows as above).    Let $\phi = M(\psi)$ as before and define $\Gamma$ exactly as in \eqref{def:Gamma}.  Then from the definition of products (see Remark \ref{rmk:productswithgradient}) at each point $(x_0,y)$ we have
$$ (\psi_x + \Gamma^T \partial_y \psi, i^*_\Gamma \Hess(\psi) ) \in F.$$
and \eqref{eq:extrasssumptiononF} implies
$$ ( \psi_x  + \Gamma^T \psi_y - \partial_y \Gamma^T,  i^*_\Gamma \Hess(\psi)  +  (\partial_y \Gamma_y)^T (\partial_y \Gamma) ) \in F.$$
We then multiply this by $e^{-\psi}$ and integrate over $y$.   Observing that
$$e^{-\psi} ( \Gamma^T \psi_y - \partial_y \Gamma^T)  = - \partial_y (e^{-\psi} \Gamma^T)$$
this extra term integrates to zero, and from (\ref{eq:derivativephi},\ref{eq:Hessofpsi}) we deduce $(d_x \phi, \Hess(\phi))\in F$.
\end{remark}

\section{Prekopa Theorem: The Non-Smooth Case}\label{sec:prekopanonsmooth}

\begin{theorem}\label{thm:prekopageneral}
Let $F$ be a convex Dirichlet set and let $V\subset \mathbb R^m$ be open and convex, and suppose $\psi:X\times V\to \mathbb R\cup\{-\infty\}$ is $F\star \mathcal P$-subharmonic.  Then $M_{X\times V}(\psi)$ is $F$-subharmonic.
\end{theorem}
\begin{proof}
Notice first that part of $\psi$ being $F\star \mathcal P$-subharmonic includes that $\psi$ is upper-semicontinuous, and so by Fatou's Lemma we get that  $M_{X\times V}(\psi)$ is also upper-semicontinuous.   The proof will proceed by a number of approximations.\\ 

{\bf Step 1: } We may assume without loss of generality that $\psi$ is bounded from below.\\

To see this we may as well assume $X$ is bounded, and let $u$ be a bounded $F$-subharmonic function on $X$ (this can be done easily by selecting any $A\in F$ and setting $u(x)=x^T A x$).  Then consider $\psi_n := \max\{ \psi, u-n\}$ so $\psi_n\searrow \psi$ pointwise as $n\to \infty$.  Each $\psi_n$ is bounded from below and $F\star \mathcal P$-subharmonic, and hence if we know the theorem holds under this assumption we have that $M(\psi_n)$ is $F$-subharmonic and decreases to $M(\psi)$, implying that $M(\psi)$ is $F$-subharmonic.\\

{\bf Step 2: } We may assume in addition that $V$ is a bounded open convex set and that $\psi$ is defined on a neighbourhood of $\overline{V}$.\\

To see this write $V = \bigcup_n {V_n}$ where each $V_n$ is bounded open and convex and so that a neighbourhood of $\overline{V_n}$ is contained in $V$.   Then $M_{X\times V_n}(\psi)$ decreases to $M_V(\psi)$, so if we know each $M_{X\times V_n}(\psi)$ is $F$-subharmonic the same is true for $M_{X\times V}(\psi)$.\\

{\bf Step 3: } We may assume that $V = \mathbb R^m$ and that for $\|y\|$ sufficiently large $\psi = A\|y\|^2 + v(x)$ for some constant $A$ and bounded $F$-subharmonic function $v$.\\

Let $V\subset \mathbb R^m$ be bounded, open and convex.  By affine transformation, we may as well assume $0\in V$.  Let $\gamma_V$ be the gauge function of $V$ given by
$$\gamma_V(y) = \inf\{ \lambda : y\in \lambda v\}$$
so $\gamma_V$ is well defined,  convex,  is $O(\|y\|)$ for $\|y\|$ large and  $\gamma_V\le 1$ on $V$ and $\gamma_V>1$ on $\overline{V}^c$.     Thus we can find constants $A,B$ so that the convex function 
$$\rho = \max \{ \gamma_V - 1, A\|y\|^2 + B\}$$
satisfies $\rho\le 0$ on $V$ and $\rho>0$ on $\overline{V}^c$ and $\rho  = A\|y\|^2 + B$ for $\|y\|$ sufficiently large.  

Now consider
$$\phi_n(x,y) = n\rho(y) + u(x)$$
where $u$ is a bounded $F$-subharmonic function chosen that that $u\le \psi$ on $V$ (this is possible as we  may as well assume that $X$ is bounded, and we are already assuming that $\psi$ is bounded from below).  Then $\phi_n$ is $F\star \mathcal P$-subharmonic.  We set
$$\psi_n = \max \{ \psi, \phi_n\}.$$
Note that $\psi_n$ is well defined on all of $X\times \mathbb R^m$ for $n$ sufficiently large (since $\psi$ is assumed to be defined on a neighbourhood of $\overline{V}$ and just outside $\overline{V}$ we will have $\phi_n>\psi$ for $n$ sufficiently large).  Moreover $\psi_n$ is $F\star \mathcal P$-subharmonic.    Thus $\psi_n$ satisfies the additional hypotheses described in Step 3, so if the result holds under these hypotheses then $M(\psi_n)$ is $F$-subharmonic.  

Clearly $\psi_n$ increases to $\psi$ on $X\times \overline{V}$ and to $\infty$ on $X\times \mathbb R^m\setminus \overline{V}$.  Thus $M_{X\times V}(\psi_n)$ increases to $M_{X\times \overline{V}}(\psi) = M_{X\times V}(\psi)$, i.e.  
$$M_{X\times V}(\psi) = \sup_n M_{X\times V}(\psi_n).$$
But as we already observed, $M_{X\times V}(\psi)$ is upper-semicontinuous and so we can replace this supremum with its upper-semicontinuous regularization.  Thus $M_{X\times V}(\psi)$ is $F$-subharmonic.\\

{\bf Step 4: } We may assume in addition that $\psi$ is continuous.\\

This follows from the general principle of sup-convolution.  Assume $X$ is a ball in $\mathbb R^n$ and $\psi$ satisfies the conditions in Step 3.  For small $\epsilon$ set
$$\psi_n(x,y) = \sup_{(x',y')\in X \times \mathbb R} \{ \psi(x',y') - \frac{1}{2\epsilon} (\|x-x'\|^2 + \|y-y'\|^2\} \text{ for } (x,y)\in \frac{1}{2}X\times \mathbb R.$$

General properties of sup-convolution imply that $\psi_n$ is $F\star \mathcal P$-subharmonic, continuous and that $\psi_n\searrow \psi$ pointwise.  In fact since $\psi = A\|y\|^2 + v(x)$ for large $\|y\|$ the same holds for each $\psi_n$ (albeit for different  $A,v$).   For the reader's convenience these properties are proved in Appendix \ref{appendix:supconvolution}.  Thus $M(\psi_n)$ decreases to $M(\psi)$ so if we know each $M(\psi_n)$ is $F$-subharmonic the same is true of $M(\psi)$.\\

{\bf Step 5: }  We may assume in addition that $\psi$ is smooth.\\

Consider the mollification $\psi_\epsilon$ of $\psi$ obtained from a mollifier that is th
Let $\alpha_0:\mathbb R\to \mathbb R$ be a smooth mollifier,  set $\alpha(x,y):= \Pi_{i,j} \alpha(x_i)\alpha(y_j)$ and consider the mollifier $\psi_\epsilon = \alpha_\epsilon * \psi$.  The $\psi_\epsilon$ is an average of translates of $\psi$ each of which remain $F\star \mathcal P$-subharmonic since $F$ is assumed to be constant-coefficient.  And as $F\star \mathcal P$ is closed and convex, and so $\psi_\epsilon$ is $F\star \mathcal P$-subharmonic.  Moreover $\psi_\epsilon$ has the same asymptotic behavior as $\psi$ for large $\|y\|$ and converges to $\psi$ locally uniformly as $\epsilon\to 0$.  This implies that $M(\psi_\epsilon)$ converges locally uniformly to $M(\psi)$, so if we know the result for such smooth functions then $M(\psi)$ will be $F$-subharmonic.\\

{\bf Step 7: } We are left now with a smooth $F\star \mathcal P$ function $\psi:X\times \mathbb R\to \mathbb R^m$ that for $\|y\|$ sufficiently large is equal to $A\|y\|^2+v(x)$ for some smooth $F$-subharmonic function $v$.         This $\psi$ certainly satisfies the condition from Theorem \ref{thm:prekopasmooth} and so for such $\psi$ we have already proved that $M(\psi)$ is $F$-subharmonic completing the proof.
\end{proof}

\section{The Brunn-Minkowski Theorem and the Minimum Principle}\label{sec:BMandminimum}

Let $X\subset \mathbb R^n$ be open and $K\subset X\times \mathbb R^m$.  We always will assume that 
 $$K_x: = \{ y \in \mathbb R^m : (x,y)\in K\}$$
  is measurable, and define
$$B_K(x):=-\log \vol(K_x).$$
Thus $B_K$ is precisely the marginal function $M_{K}(0)$ of the zero function.\\

\subsection{Semi-continuity of Marginal Functions}
The following gives some topological conditions on $K$ that ensure the marginal function is upper-semicontinuous.  
 \begin{lemma}\label{lem:semicontinuousi}
 Suppose $K\subset X\times \mathbb R^m$ is bounded and such that for each $x$
 \begin{equation}\label{eq:topologicalboundarycondition}K_x \text{ is convex with non-empty interior and } (\Int K)_x = \Int(K_x).\end{equation} 
 Suppose also  $\psi$ is upper-semicontinuous on $K$ and bounded from below and for each $x\in X$ the integral $\int_{K_x} e^{-\psi(x,y)} dy$ exists.  Then $M_K(\psi)$ is upper-semicontinuous on $X$.
 \end{lemma}
  \begin{proof}
Fix $x_0\in X$ and let $\epsilon>0$.  Fix a point in $\Int(K_{x_0})$ which we may as well take to be $0$.   Let $\epsilon>0$ and choose a  compact $V\subset \Int (K_x)$ so that $\vol(V) \ge \vol(K_{x_0}) - \epsilon$ (this is possible as $K_{x_0}$ is convex, so we can take the closure of $tK_{x_0}$ for $t$ slightly less than 1).    Then using \eqref{eq:topologicalboundarycondition} we deduce there is a ball around $x_0$ so $B\times V\subset \Int(K)$.  

 Then 
 \begin{align*}
 \liminf_{x\to x_0} \int_{K_x} e^{-\psi(x,y)} dy &\ge  \liminf_{x\to x_0} \int_{V} e^{-\psi(x,y)} dy\\
 & \ge  \int_{V}  \liminf_{x\to x_0} e^{-\psi(x,y)} dy \text{ by Fatou's Lemma } \\
  & \ge  \int_{V}  e^{-\psi(x_0,y)} dy \text{ as } \psi \text{ is upper-semicontinuous} \\
  &  = \int_{K_{x_0}} e^{-\psi(x_0,y)} dy  + \int_{K_{x_0}\setminus V} e^{-\psi(x_0,y)} dy  \\
  & \ge \int_{K_{x_0}} e^{-\psi(x_0,y)} dy - C\epsilon \text{ as } \psi \text{ is bounded from below.}
  \end{align*}
   
 Thus $x\mapsto \int_{K_x} e^{-\psi(x,y)} dy$ is lower-semicontinuous at $x_0$, and since $x_0$ was arbitrary it is lower semicontinuous on $X$.    Thus $M_{K}(\psi)$ is upper-semicontinuous.
    \end{proof}

\begin{remark}
The previous lemma is far from optimal.  For instance we can produce examples of unbounded sets $K$ for which $M_K(\psi)$ is upper-semicontinuous as follows.   Suppose that $K_n$ is an increasing sequence of subset of $X\times \mathbb R^m$ and suppose $\psi$ is defined on $K:=\bigcup K$ and for each $x$ that $\int_{K_x} e^{-\psi(x,y)} dy$ exits.     Then $M_{K_n}(\psi)$ decreases pointwise to $M_{K}$.    So if each $M_{K_n}$ for $n$ sufficiently large is upper semicontinuous on some open $U\subset X$ then $M_{K}$ is upper-semicontinuous on $U$.
\end{remark}

\subsection{The Brunn-Minkowski Theorem}

We continue to assume $X\subset \mathbb R^n$ is open.


\begin{theorem}[Brunn-Minkowski Theorem for $F$-subharmonicity]\label{thm:BMI}
Assume that $F\subset \Sym^2(\mathbb R^n)$ is convex Dirichlet set that is a cone over $0$.   Suppose that $K\subset X\times \mathbb R^m$ is closed and bounded and such that  there is an $F\star \mathcal P$-subharmonic function $\rho$ on a neighbourhood of $K$ such that $K = \{\rho \le 0\}$.  Then  the upper-semicontinuous regularization of $B_K$ is $F$-subharmonic.
\end{theorem}
\begin{proof}
Observe first that fixed $x\in X$ the function $y\mapsto \rho(x,y)$ is  convex and $K_x  = \{ y: \rho(x,y)\le 0\}$, which shows that $K_x\subset \mathbb R^m$ is a bounded convex subset of $\mathbb R^m$.  

Say that $\rho$ is defined on an open $W$ continaining $K$ and fix $x_0\in X$.  Then $K_{x_0}\subset W_{x_0}\subset \mathbb R^m$.  Since $K_{x_0}$ is compact we can find a convex $V_{x_0}$ so $K_{x_0} \subset V_{x_0}\subset W_{x_0}$,  Then as $K$ is closed and bounded we deduce that $K_{x}\subset V_{x_0}$ for all $x$ in some small ball around $x_0$.      Thus, since we can shrink $X$, there is no loss in assuming that there is a convex set $V\subset \mathbb R^m$ such that $\rho$ is defined on $X\times V$.

Now consider
$$\psi_n = \max\{ n\rho, 0\}$$
which is $F\star \mathcal P$-subharmonic.  

Thus by Prekopa's Theorem (Theorem \ref{thm:prekopageneral}) $M_{X\times V}(\psi_n)$ is $F$-subharmonic.  But $\psi_n$ is identically 0 on $K$ and increases to $\infty$ on $X\times V\setminus K$ as $n$ increases,  and thus $M(\psi_n)$ increases to $B_K$.    Hence the upper semicontinuous regularization of $B_K$  is $F$-subharmonic.
 \end{proof}

\begin{proof}[Proof of Theorem \ref{thm:BMI:intro}]
Combine Lemma \ref{lem:semicontinuousi} and Theorem \ref{thm:BMI} since the former guarantees that $B_K$ is upper-semicontinuous.
\end{proof}

\begin{remark}
\begin{enumerate}

\item We do not know if for such sets $K$ the function $B_K$ is automatically upper-semicontinuous,  or what reasonable assumptions on $F$ and $K$ will guarantee this.  When $F=\mathcal P$ this holds as then $K$ is convex, and we will later see examples where this holds when $F=F_{sub}$.
\item This Brunn-Minkowski theorem extends easily to other sets in the following way.  Suppose $K=\bigcup K_n$ is bounded where $K_n$ are increasing.   Assume each $B_{K_n}$ is $F$-subharmonic on $X$ (instance $K_n$ could satisfy the hypothesis of the above theorem).  Then $B_K$ is $F$-subharmonic on $X$  (for locally in this set we have that $B_{K_n}$ decreases pointwise to $B_{K}$).

\item As is clear from the proof, we actually have another version of Prekopa's theorem namely under the same assumptions on $K$,  if $\psi$ is $F\star \mathcal P$-subharmonic on a neighbourhood of $K$ then $M_K(\psi)$ is $F$-subharmonic on $X$ (to see this run the same argument but with $\psi_n = \max\{ n\rho, \psi\}$). 

\item In particular one can replace the Euclidean volume in the Brunn-Minkowski statement with any measure that is log-concave (i.e.a measure that is is $e^{-u}$ times the Lebesgue measure for some convex function $u:\mathbb R^m\to \mathbb R)$ by using instead $\psi_n = \max \{n\rho, u\}$ in the above proof.

\end{enumerate}
\end{remark}

It is possible to relax the condition that $F$ is a cone over $0$ at the cost of a weaker conclusion.  



\begin{theorem}[Brunn-Minkowski Theorem for $F$-subharmonicity II]\label{thm:BMII}

Assume that $F\subset \Sym^2(\mathbb R^n)$ is convex Dirichlet set. Suppose that $K\subset X\times \mathbb R^m$ is closed and bounded and such that  there is an $\overrightarrow{F\star \mathcal P}$-subharmonic function $\rho$ on a neighbourhood of $K$ such that $K = \{\rho \le 0\}$.  Assume that $B_K$ is upper-semicontinuous and $v$ is $F$-subharmonic on $X$.  Then $v+B_K$ is $F$-subharmonic.

\end{theorem}
\begin{proof}
Let $v'$ be any $F$-subharmonic function on $X$.  Replacing $v$ with $\max\{ v, v'-j\}$ for $j\in \mathbb N$ shows we may without loss of generality assume that $v$ is bounded from below.

So fix a $v\in F(X)$ bounded from below and fix any $B\in F$.    Let $q(x)$ be a quadratic form on $X$ with Hessian always equal to $B$.  Then shrinking $X$ if necessary and subtracting a constant from $q$ if necessary we may arrange so $q<v$ on $X$.

By definition $\overrightarrow{F\star \mathcal P}$, there exists a $n_0$ such that at any point $(x,y)\in K$ we have $n \Hess_{(x,y)}\rho + B\in F\star \mathcal P$ for all $n\ge n_0$.  Since $K$ is compact we can take this $n$ uniformly over all $K$ implying that $n\rho +q$ is $F\star \mathcal P$-subharmonic. The function $\psi_n = \max \{ n \rho + q, v\}$ is $F\star \mathcal P$-subharmonic.  Observe that $\psi_n$ increases pointwise to $\infty$ on $K^c$ and  identically $v$ on $K$.   The proof now proceeds as in that for Theorem \ref{thm:BMI} upon noticing that $M_K(v) = e^{-v(x)} \vol(K_x)$ so $-\log M_K(v) = v + B_K$.
\end{proof}

\section{The Minimum Principle}\label{sec:minimumprinciple}

That the minimum principle can be deduced from Preokopa's Theorem is well known, and for convenience we include the full statement and proof.

\begin{theorem}\label{thm:minimumprinciple}
Let $X\subset \mathbb R^n$ be open and assume that $F\subset \Sym^2(\mathbb R^n)$ is a convex Dirichlet set that is a cone over $0$.  Let $V\subset \mathbb R^m$ be convex and suppose that $\psi$ is $F\star \mathcal P$-subharmonic on $X\times V$.  Then the function 
$$x\mapsto \inf_y \psi(x,y)$$
is $F$-subharmonic.
\end{theorem}
\begin{proof}
This follows the argument given in \cite{Berndtsson_Prekopa}.  Replacing $\psi$ with $\max\{\psi,-j\}$ and letting $j\to \infty$ we may assume that $\psi$ is bounded from below, and then adding a constant we may assume $\psi\ge 0$.  Now for $p\ge 1$ set $\psi_p: = \psi + \frac{1}{p} |y|^2$ (which is again $F\star\mathcal P$-subharmonic since $F$ is assumed to be a cone) and put $p\psi_p$ into Prekopa's theorem to deduce that $x\mapsto \log M(p\psi_p)$ is $F$-subharmonic where
$$M(p\psi_p) = -\log \int_{\mathbb R^m} e^{-p\psi(x,y) - |y|^2}dy = -\frac{1}{p}\log \| e^{-\psi_p}\|_{L^p}$$
(and we observe this is is finite for each $x$  even if $V$ is unbounded due to the $|y|^2$ term).    Then using that the $L^p$ norm converges to the $L^{\infty}$ norm as $p\to \infty$ we conclude that $pM(p\psi_p)$ (which is still $F$-subharmonic) decreases to $-\log \|e^{-\psi}\|_{L^{\infty}} = \inf_y \psi$ as $p\to \infty$, completing the proof.
\end{proof}

\section{A Simple Explicit Example}\label{sec:example}

With variables $(x_1,x_2)\in \mathbb R^2$ and $y\in \mathbb R$, consider then the quadratic
$$\psi(x_1,x_2,y) = (x_1,x_2,y) \begin{pmatrix} \lambda & \tau & a \\ \tau & \mu & b \\ a &b &1 \end{pmatrix} \begin{pmatrix} x_1 \\x_2 \\y \end{pmatrix}.$$

Observing that $\psi_{yy} = 1$, by Proposition \ref{prop:characterization1} we know $\psi$ is $F_{sub}\star \mathcal P$ subharmonic if and only if 
$$0\le \tr( \psi_{xx} - \psi_{xy}^T\psi_{yy}^{-1} \psi_{xy}) = \tr\left( \begin{pmatrix} \lambda & \tau \\ \tau &\mu\end{pmatrix} - \begin{pmatrix}a^2 &ab\\ ab & b^2 \end{pmatrix}\right)$$
that is, if and only if, 
\begin{equation} \lambda + \mu - a^2 -b^2\ge 0.\label{ex:subharmonicproductexplicit}\end{equation}

Now fix $\kappa>0$ and set
$$K : = \{ \psi\le \kappa\} \subset \mathbb R^3.$$
Let $\pi:\mathbb R^3\to \mathbb R^2$ be the projection to the $(x_1,x_2)$ plane.  Then for each $(x_1,x_2)\in \pi(K)$ the slice $K_{(x_1,x_2)}$ is an interval $[y^{-}, y^+]$ given by the roots of the quadratic $y\mapsto \psi(x_1,x_2,y)$.  Explicitly
$$y^{\pm} = -(ax_1 + bx_2) \pm \sqrt{ (ax_1 +bx_2)^2 - (\lambda x_1^2 + \mu x_2^2 + 2\tau x_1 x_2 - \kappa)}$$
and so $V(x_1,x_2) := \vol(K_{(x_1,x_2)})$ is given by
$$V(x_1,x_2)=y^+ - y^- = 2\sqrt{ (ax_1 +bx_2)^2 - (\lambda x_1^2 + \mu x_2^2 + 2\tau x_1 x_2 - \kappa)}.$$
Then
$$B_{K}(x_1,x_2):=-\log V(x_1,x_2) =- \log 2 - \frac{1}{2} \log  W$$
where
$$W:= (ax_1 +bx_2)^2 - (\lambda x_1^2 + \mu x_2^2 + 2\tau x_1 x_2 - \kappa).$$
Now a direct computation gives
$$2B_{x_1x_1} = W^{-2} (W_{x_1})^2-W^{-1} W_{x_1x_1}   = W^{-1} (\lambda-a^2) + W^{-2}  (W_{x_1})^2$$
and similarly for the $x_2$ variable, and thus
$$ \Delta B_{K} = \frac{1}{2W}  (\lambda + \mu -a^2-b^2)+ \frac{1}{2}W^{-2} ((W_{x_1})^2 + (W_{x_2})^2).$$
Since the second term in this expression is automatically positive, we see that $\psi$ being $F\star \mathcal P$-subharmonic (which is equivalent to \eqref{ex:subharmonicproductexplicit}) implies that $B_K$ is subharmonic as expected from the Brunn-Minkowski Theorem.


\section{Prekopa's Theorem for Plurisubharmonic Functions}\label{sec:Bo}

We now show our version of Prekopa's Theorem proves the following statement for plurisubharmonic functions, originally due to Berndtsson \cite{Berndtsson_Prekopa}.

\begin{theorem}
Let $X\subseteq \mathbb{C}^n$ be open and assume that $\phi$ is plurisubharmonic on $X\times \mathbb{C}^m$. Let $w_j=x_j+iy_j$ be the coordinates on $\mathbb{C}^m$ and assume that  $\psi(z,x+iy)=\psi(z,x)$.  Then $$z\mapsto -\log\int_{\mathbb{R}^m}e^{-\psi(z,x)}dx$$ is plurisubharmonic in $X$.  
\end{theorem}

\begin{proof}
First we note that since a function is plurisubharmonic if and only if it is subharmonic along all complex lines, if it is enough to consider $X\subseteq \mathbb{C}$. Another classical fact is that a function on $\mathbb{C}^m$ which is independent of $y$ is plurisubharmonic if and only if it is convex in $x$, so our hypotheses imply $\psi(z,x)$ is convex in $x$. 

Now let $\Gamma=\Gamma_0+p$ be an affine map from $\mathbb{C}$ to $\mathbb{R}^m$. It is easy to see that the linear map $\Gamma_0$ can be extended to a complex linear map $\Gamma_{0,\mathbb{C}}:\mathbb{C}\to \mathbb{C}^m$, and so $\Gamma$ extends to the complex affine map $\Gamma_{\mathbb{C}}:=\Gamma_{0,\mathbb{C}}+p$. Since $\phi$ is plurisubharmonic and independent of $y$ we get that $\psi(z,\Gamma(z))=\psi(z,\Gamma_{\mathbb{C}}(z))$ is subharmonic in $z$. We thus see that $\psi(z,x)$ is $F_{sub}\star \mathcal P$-subharmonic so our Prekopa Theorem applies, and the result follows.  \end{proof}

\section{Interpolation}\label{sec:interpolation}

\subsection{The Perron Envelope}

We are now ready to discuss the application to interpolation of convex functions and to interpolation of convex sets.   Let $X\subset\mathbb R^n$ be open, and $F$ be a convex Dirichlet set.   Suppose $\Omega\subset X$ is open and smooth bounded, and that we are given a continuous
$$\phi:\partial \Omega \times \mathbb R^m\to \mathbb R$$
such that $\phi_\tau(\cdot) = \phi(\tau,\cdot)$ is convex for each fixed $\tau\in \partial \Omega$.  The goal is to interpolate this to a family of convex functions on $\Omega\times \mathbb R^m$, which we achieve through the following envelope

\begin{equation}\Phi:= \sup \{ \zeta \in USC(\overline{\Omega} \times \mathbb R^m): \zeta|_{\Omega\times \mathbb R^m} \text{ is } F\star\mathcal P\text{-subharmonic and }\zeta|_{\partial \Omega\times \mathbb R^m} \le \phi \}\label{eq:perron}
\end{equation}
where $USC(X)$ denotes the set of upper-semicontinuous functions $X\to \mathbb R\cup \{-\infty\}$.

For any convex function $\eta$ we will write $\eta^*$ for its Legendre transform
$$\eta^*(u) = \sup_{y\in \mathbb R^m} \{ y\cdot u - \eta(y)\}$$

\begin{definition} \label{def:locallycomparable}We say that $\{\phi^*_{\tau}\}_{\tau \in \partial \Omega}$ is \emph{locally comparable} if 
\begin{enumerate}
\item for all $u_0\in \mathbb R^m$  we have $|\phi^*_{\tau}(u)-\phi^*_{\tau'}(u)|\le C$ uniformly over $\tau,\tau'\in \partial \Omega$ (in this definition we allow the possibility that $\phi^*_\tau(u)$ is infinite in which case the condition requires that $\phi_{\tau'}(u)$ is infinite for all $\tau'$).
\item for each $\tau\in \partial \Omega$ the set of $u$ such that $\phi_{\tau}^*$ is finite is open in $\mathbb R^m$ (and by (1) this set is independent of $\tau$)
\end{enumerate}
\end{definition}

\begin{proposition}\label{prop:dirichletinterpolationbasics}
Assume $\Omega$ is strictly $\overrightarrow{F}$-convex and strictly $\overrightarrow{\tilde{F}}$-convex and that $\{\phi^*_{\tau}\}_{\tau \in \partial \Omega}$ is locally comparable.   Then
$\Phi$ is $F\star \mathcal P$ subharmonic, and $\Phi|_{\partial \Omega\times \mathbb R^m} = \phi$.
\end{proposition}
\begin{proof}
This is a standard proof and will follow \cite[\S 6]{HL_Dirichletduality} closely (in fact the only reason that \cite[\S 6]{HL_Dirichletduality}  does not apply directly is that we are working on the unbounded set $\Omega\times\mathbb R^m$).



%


In detail, since we are assuming that $\Omega$ is strictly $\overrightarrow{\tilde{F}}$-convex, there is a globally defined $\overrightarrow{\tilde{F}}$-subharmonic defining function $\rho$ on $\overline{\Omega}$.   It is shown in \cite[(5.4)]{HL_Dirichletduality} that there exists an $\epsilon>0$ and $R>0$ such that $C(\rho - \epsilon|\tau-x|^2)$ is $\tilde{F}$-subharmonic on $\overline{\Omega}$ if $C\ge R$.  

Now fix $(\tau_0,y_0)\in \partial\Omega \times \mathbb R^m$ and let $\delta>0$ and $B$ be a ball around $y_0$.   By continuity of $\phi$ we can choose $C$  large  enough so that
$$\phi(\tau,y) + C(\rho - \epsilon|\tau-\tau_0|^2) = \phi(\tau,y) - C \epsilon|\tau-\tau_0|^2 \le \phi(\tau_0,y) + \delta \text{ for }  (\tau,y)\in \partial \Omega\times B.$$
Now let $\zeta$ be a function as in the right hand side of \eqref{eq:perron} and set $$w = \zeta + C(\rho - \epsilon |x-\tau_0|^2)$$ which is subaffine on $\Omega\times \mathbb R^m$ and upper-semicontinuous on $\overline{\Omega}\times \mathbb R^m$.  By the maximum principle, for each $y\in \mathbb R^m$ we have $\sup_{\overline{\Omega}\times \{y\}} w = \sup_{\partial \Omega \times \{y\}}w$.   So applying this for each $y\in B$ and using the previous inequality we have
$$\sup_{\overline{\Omega} \times B} w   = \sup_{\partial\Omega\times B} w  \le \phi(\tau_0,y) + \delta$$
or said another way
$$\zeta(x,y) + C (\rho - \epsilon |x-\tau_0|^2) \le \phi(\tau_0,y) + \delta \text{ for } (x,y)\in \overline{\Omega}\times B.$$
Taking the supremum over all such $\zeta$ yields
$$\Phi(x,y)+ C (\rho - \epsilon |x-\tau_0|^2) \le \phi(\tau_0,y) + \delta \text{ for } (x,y)\in \overline{\Omega}\times B$$
Now let $\Theta$ be the upper-semicontinuous regularization of $\Phi$.   So taking the upper semicontinuous regularization
$$\Theta(x,y)+ C (\rho - \epsilon |x-\tau_0|^2) \le \phi(\tau_0,y) + \delta \text{ for } (x,y)\in \overline{\Omega}\times B.$$
In particular, evaluating at $(\tau_0,y_0)$ gives $\Theta(\tau_0,y_0)\le \phi(\tau_0,y_0)+\delta$ and so letting $\delta$ tend to zero we in fact have $\Theta(\tau_0,y_0)\le \phi(\tau_0,y_0)$.  Since $\tau_0,y_0$ are arbitrary, this means $\Theta$ is a candidate in the right hand side of \eqref{eq:perron} from we deduce that $\Phi = \Theta$, so $\Phi$ is in fact upper-semicontinuous.    Thus $\Phi$ is $F\star \mathcal P$-subharmonic on $\Omega\times \mathbb R^m$ and satisfies $\Phi\le \phi$ on $\partial \Omega\times \mathbb R^m$.

To prove the other inequality fix $\tau_0\in \partial \Omega$ and $y_0\in \mathbb R^m$.  Since $\Omega$ is assumed to be strictly $F$-convex, there is a $F$-subharmonic defining function $\rho$ and an $\epsilon>0$ such that for $C\gg 0$ the function $C(\rho - \epsilon |x-\tau_0|^2)$ is $F$-subharmonic.   Fix $\delta>0$ and let $u_0\in \mathbb R^m$ be such that $\phi(\tau_0,y) \ge u_0\cdot (y-y_0) + \phi_{\tau_0}(y_0)$ (i.e. $u_0$ is a subgradient of $\phi_{\tau_0}$ at $y_0$) and set $$g(y): = u_0\cdot (y-y_0) + \phi_{\tau_0}(y_0).$$

We claim that for $C\gg 0$ we have
\begin{equation} g(y) + C(\rho - \epsilon |\tau-\tau_0|^2) \le \phi_\tau(y) + \delta \text{ for } \tau\in \partial \Omega, y\in \mathbb R^m.\label{eq:finalclaim}\end{equation}
Given this for now set
$$ v(y): = g(y) + C(\rho - \epsilon |\tau-\tau_0|^2)- \delta$$
Then $v$ is $F\star \mathcal P$-subharmonic on $\Omega\times \mathbb R^m$ and $v_{\partial \Omega\times \mathbb R^m} \le \phi$.  So $v$ is a candidate for the envelope defining $\Phi$ and thus $\Phi\ge v$.  Evaluating this at $(\tau_0,y_0$ gives
$$\Phi(\tau_0,y_0) \ge v(\tau_0, y_0) = \phi(\tau_0,y_0)-\delta.$$
Letting $\delta\to 0$ this shows $\Phi|_{\partial \Omega\times \mathbb R^m} \ge \phi$.

For the claim, note first that by our choice of $u_0$ we have $\phi_{\tau_0}^*(u_0) + \phi_{\tau_0}(y_0) = y_0\cdot u_0$.   Now the hypothesis that $\{\phi_{\tau}^*\}_{\tau\in \partial \Omega}$ is locally comparable means that there is a $C'$  such that $\phi_{\tau}^*(u_0)\le \phi_{\tau_0}^*(u_0) +C'$ for all $\tau\in \partial \Omega$.    Thus from the definition of the Legendre transform this implies that for all $y$ we have $$y\cdot u_0 - \phi_{\tau}(y) \le  \phi_{\tau_0}^*(u_0) + C' = y_0\cdot u_0 - \phi_{\tau_0}(y_0)+C'$$ which rearranges to $\phi_{\tau}(y) \ge   g(y) - C'$.    This means that as long as $\tau$ is a bounded distance away from $\tau_0$ we can arrange so that \eqref{eq:finalclaim} holds for $C$ sufficiently large.

So it remains to consider $\tau$ close to $\tau_0$.   By definition $u_0$ is a subgradient of $\phi_{\tau_0}$ at $y_0$, and using Definition \ref{def:locallycomparable}(1) for each $\tau$ there is a $y_{\tau}\in \mathbb R^m$ so that $u_0$ is a subgradient of $\phi_\tau$ at $y_{\tau}$.  Then if \eqref{eq:finalclaim} holds at $(\tau,y_\tau)$ then it holds at $(\tau,y)$ for all $y$.   

We claim there is ball $B$ around $y_0$ such that $y_\tau\in B$ for all $\tau$ close to $\tau_0$.   But this is clear, for we know that $\phi^*_{\tau}(u_0)$ is finite, and hence by Definition \ref{def:locallycomparable}(2) this holds also for $u$ close to $u_0$.  Hence $\phi_{\tau_0}-g$ tends gets large when $\|y\|$ tends to infinity.  In particular there is a large ball $B$ around $y_0$ such that $\phi_{\tau_0}-g$ is strictly positive on $\partial B$.  Then by continuity of $(\tau,y)\mapsto \phi_{\tau}(y)$ we get that $\phi_{\tau}-g$ is strictly positive on $\partial B$ for $\tau$ sufficiently close to $\tau_0$.  But this implies that $u_0$ is a subgradient of $\phi_{\tau_0}$ for some $y_\tau\in B$.

Thus we need only show  \eqref{eq:finalclaim} holds for $\tau$ close to $\tau_0$ and $y$ in some given ball $B$ around $y_0$.  But it is clear that \eqref{eq:finalclaim} holds at $\tau_0$ (for all $y)$ and so by continuity this holds for $\tau$ close to $\tau_0$ and all $y$ in $B$ completing the proof.



\end{proof}

\subsection{Legendre Duality}
There is a useful dual description we can make for this interpolation.   For a function $\Phi(\tau,y)$ in two variables we write $\Phi^*$ for the Legendre transform in the second variable, so 
$$\Phi^*(x,u) = \sup_{y\in \mathbb R^m}\{ y\cdot u - \Phi(x,y)\}.$$

\begin{proposition}\label{prop:interpolationduality}Assume $\Omega$ is strictly $\overrightarrow{F}$-convex and strictly $\overrightarrow{\tilde{F}}$-convex. Then the envelope $\Phi$ from \eqref{eq:perron} satisfies
\begin{equation}-\Phi^*(x,u) = {\sup}\{ f\in USC(\overline{\Omega}) : f\in F(\Omega) \text{ and }   f(\tau) \le - {\phi^*_\tau}(u) \text{ for } \tau\in \partial \Omega \}\label{eq:dualequation}\end{equation}
\end{proposition}
\begin{proof}

This is essentially tautological given our minimum principle.  Let $S$ denote the right hand side of \eqref{eq:dualequation}.
  For fixed $u$ the function $\Phi(x,y) - y\cdot u$ is $F\star \mathcal P$-subharmonic, so by the minimum principle (Theorem \ref{thm:minimumprinciple}) $\alpha(x):=\inf_y \{ \Phi(x,y) - y\cdot u\}$ is $F$-subharmonic on $\Omega$. One checks also that $\alpha$ is upper-semicontinuous on $\overline{\Omega}$ and using that $\Phi|_{\partial \Omega}\times \mathbb R^m\le \phi$ (Theorem \ref{prop:dirichletinterpolationbasics}) it is easy to verify  that $\alpha(x) \le - \phi^*(x,u)$.  Thus $-\Phi^*(x,u) =- \sup_y \{ - \Phi(x,y) +y\cdot u\} = \inf_y \{ \Phi(x,y) - y\cdot u\} \le \alpha(x)\le S$ giving one inequality.
  
The other inequality is elementary.  Fix $u$, $f$ be as in the right hand side of \eqref{eq:dualequation} and set $\zeta(x,y) := f(x,y) + y\cdot u$.  Then $\zeta$ is $F\star \mathcal P$-subharmonic on $\Omega\times \mathbb R^m$ and upper-semicontinuous on $\overline{\Omega}\times \mathbb R^m$ and $f \le \-\phi^*_{\tau}(u)$ over $\partial \Omega$ implies $\zeta\le \phi$ on $\partial \Omega\times \mathbb R^m$. This $\Phi\ge \zeta= f(x,y) + y\cdot u$, so $-f(x,y) \ge -\Phi(x,y) + y\cdot u$.  Taking the supremum over all $y$ gives $-f(x)\ge \Phi^*(x,u)$ and then taking the infimum over all such $f$ yields $-S \ge \Phi^*$ completing the proof.
\end{proof}

So said another way, Proposition \ref{prop:interpolationduality} shows that the data for ${\Phi}^*$ is contained in a family (parametrized by $u$) of Dirichlet problems for $F$-subharmonicity on $\Omega$.  Moreover since taking the partial Legendre transform is an involution, we see that this family of Dirichlet problems contains enough information to recover $\Phi$.\\

Consider next above picture in the special case that $F=F_{sub}$, so $F_{sub}$-sub\-harmonicity the usual notion of subharmonicity.  Then the assumption on the $F$-subharmonicity of $\Omega$ is vacuous by the following statement.

\begin{lemma} Any open bounded $\Omega\subset \mathbb R^n$ with smooth boundary is strictly $F_{sub}$-convex and strictly $\widetilde{F_{sub}}$-convex.
\end{lemma}

\begin{proof}
Since $\widetilde{F_{sub}} = F_{sub}$ we need only prove the first statement.  Let $f$ be a smooth strictly convex function on $\mathbb R$ with $f(0)=0$ and $f(t)>0$ for $t>0$ and $f(t)<0$ for $t<0$.  Let $U$ be a tubular neighbourhood of $\partial \Omega$ and $\delta(x):U\to \mathbb R$ be the signed distance to $\partial \Omega$ chosen to be negative inside $\Omega$ and positive outside.    Then for $\lambda>0$ one can compute the Hessian of $\rho:=f(\lambda \delta(x))$ and see that this is strictly subharmonic near $\partial \Omega$.  Moreover $\rho$ is strictly negative just inside $\partial \Omega$,  so taking $\max\{ \rho, \epsilon_1 |x|^2-\epsilon\}$ for small $\epsilon_i$ we get a strictly subharmonic defining function for $\Omega$. 
\end{proof}

Turning back to the interpolation for each $\tau\in \Omega$ there exists the  \emph{harmonic measure} $d\omega_\tau$ on $\partial \Omega$ with the following property: if $\phi$ is a continuous function on $\Omega$ then the solution to the Dirichlet problem
$$ -\Delta f = 0 \text{ and } f= \phi \text{ on } \partial \Omega$$
is given by
$$f(\tau) = \int_{\partial \Omega} \phi(y) d\omega_\tau(y).$$
\begin{corollary}\label{cor:harmonicinterpolationconvexfunction}
With $\Phi$ defined as in \eqref{eq:perron}  and $F=F_{sub}$ and $\Omega$ open bounded with smooth boundary we have
$$ {\Phi^*}(x,u) = \int_{\partial \Omega} \phi^*_\tau(u) d\omega_x(\tau)$$
\end{corollary}
\begin{proof}
The right hand side of \eqref{eq:dualequation} is the solution of the Dirichlet problem for the Laplacian with boundary data $-\phi^*(x,u)$.  Thus Proposition \ref{prop:interpolationduality} gives the result we want.
\end{proof}

\subsection{Interpolation of Convex Sets}

We now apply this to interpolation of convex sets.  
Given a convex set $A\subset \mathbb R^m$ let 
$$\chi_A = \left\{ \begin{array}{cc} 0 & y\in A \\ \infty &y\notin A\end{array}\right.$$
denote its characteristic function, which is a convex function on $\mathbb R^m$ (in the extended sense).  Its Legendre transform is called the support function of $A$ and is given by

$$ h_{A}(y) := \sup_{z\in A_\tau}\{ y\cdot z\}$$


Suppose once again $\Omega\subset \mathbb R^n$ is assumed to be bounded with smooth boundary and $F\subset \Sym^2(\mathbb R^n)$ is assumed to be a Dirichlet set.  Suppose that for each $\tau\in \partial \Omega$ we are given a non-empty compact convex set $A_{\tau}\subset \mathbb R^m$ that varies continuously with respect to the Hausdorff metric.   For  $k\in \mathbb N$ consider the function
$$\phi_{\tau,k} : = \chi_{A_\tau,k}(x)= e^{dist(x,A_\tau)} -1$$
which one observes is continuous with respect to $\tau$ and for each fixed $\tau$ $\phi_{\tau,k}$ increases pointwise to $\chi_{A_{\tau}}$ as $k$ tends to infinity.

Now let $F\subset  \Sym^2(\mathbb R^n)$ be a convex Dirichlet set.  Then the corresponding envelope is
\begin{equation}\Phi_k:= \sup \{ \zeta \in USC(\overline{\Omega} \times \mathbb R^m): \zeta|_{\Omega\times \mathbb R^m} \text{ is } F\star\mathcal P\text{-subharmonic and }\zeta|_{\partial \Omega\times \mathbb R^m} \le \phi_{\tau,k} \}\label{eq:perron:repeated}
\end{equation}
We set
$$\Phi : = \sup_k \Phi_k:\Omega\times \mathbb R^m
\to \mathbb R\cup \{\infty\}$$

\begin{definition}\label{eq:setinterpolation}
The $F$-interpolation of the data $\{A_{\tau}\}$ is given by setting
$$A_{x} = \{ y\in \mathbb R^m: \Phi(x,y)\le 0\}$$
\end{definition}

\begin{theorem}
Each $A_x$ is a convex subset of $\mathbb R^m$ that agrees with the given boundary data.  
 \end{theorem}
 \begin{proof}
 For fixed $x$ the function $y\mapsto \Phi_k(x,y)$ is convex, and hence the same is true of $y\mapsto \Phi(x,y)$ (in the extended sense) so $A_x$ is convex.  It agree with the boundary given boundary data since $\Phi_k|_{\partial \Omega\times \mathbb R^m} = \phi_{\tau,k}$ and so $\Phi|_{\partial \Omega\times \mathbb R^m} = \chi_{A_\tau}$.  
 \end{proof}

When $F=\mathcal P$ the hypothesis on $\Omega$ is that is be strictly convex.  Then the above interpolation agrees with the naive interpolation obtained by taking the convex hull of $\cup_{\tau\in \partial \Omega} \{\tau\} \times A_{\tau}$.  When $F=F_{sub}$ we get the harmonic interpolation from the introduction as we will describe next.

\subsection{Harmonic Interpolation}

We specialize to the case $F=F_{sub}$.  With $\Omega$ and $\{A_\tau\}$ as in the previous section, consider the integral
$$S: = \int_{\partial \Omega} A_{\tau} d\omega$$
which will again be a convex subset of $\mathbb R^m$.  This can be thought of in a number of ways, for instance since we will always assume $A_{\tau}$ varies continuously (with respect to the Hausdorff metric say) we may as well consider this as a limit of Riemann sums weighted with respect to the measure $d\omega$ with the Riemann sum being understood as taken with respect to  Minkowski addition (see \cite{Aubin_setvaluedanalyais} for some background into the integral of set-valued functions)   In any case since the function $A\mapsto h_A$ is linear (with respect to Minkowski addition) we have
$$h_{S} = \int_{\partial \Omega} h_{A_{\tau}} d\omega$$

\begin{proposition}\label{prop:harmonicasenvelope}
Let $F=F_{sub}$ and $\Omega\subset \mathbb R^m$.  Then the $F$-interpolation of $\{A_\tau\}$ agrees with the harmonic interpolation defined in the introduction.  That is
$$A_x = \int_{\partial \Omega} A_{\tau} d\omega_{x}(\tau)$$
In particular for the harmonic interpolation, the function $x\mapsto -\log \vol(A_x)$ is subharmonic.
\end{proposition}
\begin{proof}
Set $B_x =  \int_{\partial \Omega} A_{\tau} d\omega_{x}(\tau)$.  Applying Corollary \ref{cor:harmonicinterpolationconvexfunction} for for each $k$ gives
$$\Phi_k^* = \int_{\partial \Omega} \phi_{\tau,k}^*(u) d\omega_x(\tau)$$
As $k$ increases the function $\phi_{\tau,k}$ increase to $\chi_{A_\tau}$ pointwise, and so $\phi_{\tau,k}^*$ decrease pointwise to $h_{A_\tau}$.  Furthermore $\Phi_k$ increases to $\Phi$ pointwise and so $\Phi_{k}^*$ decreases to $\Phi^*$.  Thus taking the limit as $k$ tends to infinity and using the monotone convergence theorem
$$\Phi^*(x,u) = \int_{\partial \Omega} h_{A_{\tau}} (u) d\omega_x(\tau) = h_{B_x}$$
 So taking the Legendre transform
$$\Phi(x,y) = \chi_{B_x}(y)$$
Thus $A_x = \{ \chi_{B_x}\le 0\}$ which implies $A_x=B_x$ as claimed.  Finally using that the harmonic measure is continuous with respect to $x$ shows that $x\mapsto -\log \vol(A_x)$ is continuous, and hence it is $F$-subharmonic by our Brunn-Minkowskii Theorem.
\end{proof}

\appendix

\section{Sup-convolution}\label{appendix:supconvolution}

The following sup-convolution construction and its properties are essentially standard.  We will need them for possibly unbounded functions that have prescribed behavior at infinity. 

\begin{proposition}\label{prop:supconvolution}
Let $X\subset \mathbb R^n$ be an open ball and $F$ a Dirichlet set that is a cone over $0$.  Let  $\psi$  be an $F\star \mathcal P$-subharmonic function defined in a neighbourhood of $\overline{X}\times \mathbb R$ that is bounded from below suppose and such that there are constants $R,A$ and a bounded $F$-subharmonic function $v$ on $X$ with
$$ \psi(x,y) = A \|y\|^2 + v(x) \text{ for } \|y\|>R.$$
For $\epsilon>0$ define
$$\psi_\epsilon(x,y) = \sup_{(x',y')\in X\times \mathbb R} \{ \psi(x',y) - \frac{1}{2\epsilon} \|x-x'\|^2 - \frac{1}{2\epsilon} \|y-y'\|^2\} \text{ for } (x,y) \in \frac{1}{2}X \times \mathbb R.$$
Then for $\epsilon$ sufficiently small,
\begin{enumerate}
\item $\psi_\epsilon$ is continuous. 
\item $\psi_{\epsilon}(x,y) = A_\epsilon \|y\|^2 +{v}_\epsilon(x)  \text{ for } \|y\|\ge R+1$
for some positive constant $A_\epsilon$ and some $F$-subharmonic function $v_{\epsilon}$
\item $\psi_\epsilon\searrow \psi$ pointwise as $\epsilon\to 0$.
\item $\psi_\epsilon$ is $F\star \mathcal P$-subharmonic. 
\end{enumerate}
\end{proposition}
\begin{proof}
If $(x,y) \in \frac{1}{2}X \times \mathbb R$ then $\psi_\epsilon(x,y)  + \frac{1}{2\epsilon} \|x\|^2 + \frac{1}{2\epsilon} \|y\|^2$ equals
\begin{align}
&= \sup_{(x',y')\in X\times \mathbb R} \{ \psi(x',y) - \frac{1}{2\epsilon}\left( \|x-x'\|^2 -\|x\|^2\right)- \frac{1}{2}\left( |y-y'|^2 -\|y\|^2\right)\}\\
&=\sup_{(x',y')\in X\times \mathbb R} \{ \psi(x',y) + \frac{1}{\epsilon} x\cdot x'  +  \frac{1}{\epsilon} y\cdot y'  - \frac{1}{2\epsilon} \|x'\|^2 - \frac{1}{2\epsilon} \|y'\|^2\}
\end{align}
and this is a supremum of linear functions in $(x,y)$ and thus is convex.  Hence $\psi_\epsilon(x,y)  + \frac{1}{2\epsilon} \|x\|^2 + \frac{1}{2\epsilon} \|y\|^2$ is certainly continuous, and thus so is $\psi_\epsilon$.\\

For the remaining properties set
 $$u(x',y') = \psi(x',y') - \frac{1}{2\epsilon} \|x-x'\|^2 - \frac{1}{2\epsilon} \|y-y'\|^2$$
so 
$$\psi_\epsilon = \sup_{(x',y')\in X\times \mathbb R} u(x',y').$$
As $\psi$ is bounded from below and uppersemicontinuous  we can fix an $M$ so that $|\psi(x,y)|<M$ for $\|y\|\le R+2$ and $x\in X$.   Observe also that certainly $\psi_{\epsilon} \ge \psi$.\\

To prove (2) fix $\|y\|\ge R+1$.   Then any $\|y'\|\le R$  has $\|y-y'\|\ge 1$ and so for  $\epsilon$ sufficiently small $$u(x',y') \le M - \frac{1}{2\epsilon} \le AR^2 + v(x) \le A \|y\|^2 + v(x) = \psi(x,y)\le \psi_\epsilon(x,y).$$
 On the other hand for $\|y'\|\ge R$ we have $$u(x',y') = A\|y'\|^2 + v(x) - \frac{1}{2\epsilon} \|x-x'\|^2 - \frac{1}{2\epsilon} \|y-y'\|^2.$$  For $\epsilon$ sufficiently small this is concave in $y'$ and so has a unique maximum that occurs at some point (that depends on $y$).   Thus the maximum of $u(x',y')$ can be calculated by elementary means and results in
 \begin{equation}\psi_\epsilon(x,y) = A_{\epsilon} \|y\|^2 + \tilde{v}(x) \text{ for } \|y\|\ge R+1 \label{eq:psiepsilonlagey}\end{equation}
where $A_\epsilon$ is a positive constant for $\epsilon$ sufficiently small, and $\tilde{v}(x) = \sup_{x'\in X} v(x) - \frac{1}{2\epsilon}\|x-x'\|^2$ which is $F$-subharmonic.   This proves (2).  \\

{\bf Claim: } Set $\delta:= \sqrt{4M\epsilon}$.  Then for $\|y\|\le R+1$ and $\epsilon$ sufficiently small 
\begin{equation}\psi_\epsilon = \sup_{\|x-x'\|< \delta, \|y-y'\|\le \delta}  u(x',y') \label{eq:supconvolutionislocal}\end{equation}

To show this suppose $\|x-x'\|\ge \delta$ or $\|y-y'\|\ge \delta$.  If $\|y'\|\le R+2$ then $u(x',y') \le M - \frac{\delta^2}{2\epsilon} \le -M \le \psi(x,y)\le \psi_\epsilon(x,y)$.  And if $\|y'\|\ge R+2$ then $\|y-y'\|\ge 1$ and so $u(x',y')\le M - \frac{1}{2\epsilon} \le  -M\le \psi(x,y)\le \psi_\epsilon(x,y)$ for $\epsilon\ll 1$.   This proves the claim.\\

Now \eqref{eq:psiepsilonlagey} certainly implies $\psi_\epsilon \searrow \psi$ as $\epsilon\to 0$ for $\|y\|\ge R+1$.  And \eqref{eq:supconvolutionislocal} and the fact that $\psi$ is upper-semicontinuous and $\delta\to 0$ as $\epsilon\to 0$ we conclude that $\psi_{\epsilon} \searrow \psi$ on $\|y\|\le R+1$ as well.  This proves (3).\\

Finally (4) follows easily since for $\|y\|> R$ then $\psi$ is a convex function of $y$ so it is $F\star \mathcal P$-subharmonic there.   On the other hand for $\|y\|\le R+1$ we have $$\psi_\epsilon = \sup_{\|\tau|< \delta, \|\sigma'\|\le \delta}  \psi(x+\tau, y+\sigma) - \frac{1}{2\epsilon} \|\tau\|^2 - \frac{1}{2\epsilon} \|s\|^2$$
Since $F$ is constant coefficient (and independent of the constant part) the function $(x,y)\mapsto  \psi(x+\tau, y+\sigma) - \frac{1}{2\epsilon} \|\tau\|^2 - \frac{1}{2\epsilon} |s|^2$ if $F\star \mathcal P$-subharmonic.  Hence this supremum is as wel, and thus $\psi_\epsilon$ is $F\star \mathcal P$-subharmonic for $\|y\|< R+1$ as well.

\end{proof}

\let\oldaddcontentsline\addcontentsline
\renewcommand{\addcontentsline}[3]{}
\bibliographystyle{plain}
\bibliography{prekopa}{}
\let\addcontentsline\oldaddcontentsline

\medskip
\small{
\noindent {\sc Julius Ross,  Mathematics Statistics and Computer Science, University of Illinois at Chicago, Chicago  IL, USA\\ juliusro@uic.edu}\medskip

\noindent{\sc David Witt Nystr\"om, 
Department of Mathematical Sciences, Chalmers University of Technology and the University of Gothenburg, Sweden \\ wittnyst@chalmers.se, danspolitik@gmail.com}

}

\end{document}